\documentclass{article}

\usepackage{amsmath} 
\usepackage{amsthm}        
\usepackage{amssymb}       
\usepackage{braket}       
\usepackage{amsfonts}
\usepackage[all]{xy}
\usepackage{xspace}
\usepackage{calc}
\usepackage{comment}
\usepackage{pgfplots}
\usepackage{diagrams}
\usepgflibrary{fpu}
\usepackage{mathtools}
\usepackage{tabularx}
\usepackage{ifthen}             
\usepackage{graphicx}
\usepackage{docmute}
\usepackage{array}
\usepackage{subfloat} 
\usepackage{textcomp}
\usepackage{bbm}            
\usepackage{etoolbox}  
\usepackage{verbatim}
\usepackage{stmaryrd}
\usepackage{cancel}
\usepackage{xcolor}
\colorlet{Black}{black}
\usepackage{tensor}
\usepackage{enumerate}  
\usepackage{thm-restate}
\usepackage[numbers,sort]{natbib}
\usepackage{enumitem}
\usepackage{tocloft}
\usepackage{pdfpages}
\usepackage[export]{adjustbox}

\usepackage[margin=3cm]{geometry}

\newenvironment{tz}[1][]{%
                                \begin{tikzpicture}[baseline={([yshift=-.8ex]current bounding                        box.center)},#1] %
                                }{%
                        \end{tikzpicture} %
                        }

\DeclareRobustCommand{\SkipTocEntry}[5]{}

\usepackage{tikz}
\usetikzlibrary{matrix}
\usetikzlibrary{decorations.pathreplacing}
\usetikzlibrary{decorations.markings}
\usetikzlibrary{arrows}
\usetikzlibrary{calc}
\usetikzlibrary{shapes.misc}
\usetikzlibrary{fit}
\usepgflibrary{decorations.pathmorphing}
\usepgflibrary{shapes.geometric}

\pgfdeclarelayer{edgelayer}
\pgfdeclarelayer{nodelayer}
\pgfdeclarelayer{foreground}
\pgfsetlayers{edgelayer,nodelayer,main,foreground}

\tikzstyle{none}=[inner sep=0pt]

\tikzstyle{rn}=[circle,fill=Red,draw=Black,line width=0.8 pt]
\tikzstyle{gn}=[circle,fill=Lime,draw=Black,line width=0.8 pt]
\tikzstyle{bl}=[circle,fill=Blue,draw=Black,line width=0.8 pt]

\tikzstyle{simple}=[-,draw=Black,thick]
\tikzstyle{arrow}=[-,draw=Black,postaction={decorate},decoration={markings,mark=at position .5 with {\arrow{>}}},thick]
\tikzstyle{tick}=[-,draw=Black,postaction={decorate},decoration={markings,mark=at position .5 with {\draw (0,-0.1) -- (0,0.1);}},line width=2.000]
\makeatother\def\thickness{0.7pt}
\tikzstyle{dot}=[circle, draw=black, fill=black, inner sep=.5ex, line width=\thickness, node on layer=foreground]

\makeatletter
\pgfkeys{%
  /tikz/on layer/.code={
    \pgfonlayer{#1}\begingroup
    \aftergroup\endpgfonlayer
    \aftergroup\endgroup
  },
  /tikz/node on layer/.code={
     \gdef\node@@on@layer{%
      \setbox\tikz@tempbox=\hbox\bgroup\pgfonlayer{#1}\unhbox\tikz@tempbox\endpgfonlayer\egroup}
    \aftergroup\node@on@layer
  },
  /tikz/end node on layer/.code={
    \endpgfonlayer\endgroup\endgroup
  }
}

\def\node@on@layer{\aftergroup\node@@on@layer}

\makeatletter
\def\calign@preamble{%
   &\hfil\strut@
    \setboxz@h{\@lign$\m@th\displaystyle{##}$}%
    \ifmeasuring@\savefieldlength@\fi
    \set@field
    \hfil
    \tabskip\alignsep@
}
\let\cmeasure@\measure@
\patchcmd\cmeasure@{\divide\@tempcntb\tw@}{}{}{}
\patchcmd\cmeasure@{\divide\@tempcntb\tw@}{}{}{}
\patchcmd\cmeasure@{\ifodd\maxfields@
  \global\advance\maxfields@\@ne
  \fi}{}{}{}    
\newenvironment{calign}
{%
  \let\align@preamble\calign@preamble
  \let\measure@\cmeasure@
  \align
}
{%
  \endalign
}  
\makeatother
\tikzset{
    master/.style={
        execute at end picture={
            \coordinate (lower right) at (current bounding box.south east);
            \coordinate (upper left) at (current bounding box.north west);
        }
    },
    slave/.style={
        execute at end picture={
            \pgfresetboundingbox
            \path (upper left) rectangle (lower right);
        }
    }
}

\tikzset{blob/.style={draw, circle, fill=white, inner sep=1pt, minimum width=15pt, font=\scriptsize, line width=0.7pt}}
\tikzset{greenregion/.style={fill=green, fill opacity=0.3, draw=none}}
\tikzset{redregion/.style={fill=red, fill opacity=0.3, draw=none}}
\tikzset{blueregion/.style={fill=blue, fill opacity=0.3, draw=none}}
\tikzset{yellowregion/.style={fill=yellow, fill opacity=0.5, draw=none}}
\tikzset{cyanregion/.style={fill=cyan, fill opacity=0.3, draw=none}}
\tikzset{orangeregion/.style={fill=orange, fill opacity=0.6, draw=none}}
\tikzset{solidgreenregion/.style={fill=green!30, fill opacity=1, draw=none}}
\tikzset{solidredregion/.style={fill=red!30, fill opacity=1, draw=none}}
\tikzset{solidblueregion/.style={fill=blue!30, fill opacity=1, draw=none}}
\tikzset{solidyellowregion/.style={fill=yellow!30, fill opacity=1, draw=none}}
\tikzset{string/.style={line width=0.7pt}}
\tikzset{zig/.style={decoration={zigzag,segment length=3, amplitude=0.5}}}
\tikzset{bnd/.style={draw,string}}   
\tikzset{projector/.style={circle, draw, font=\scriptsize, inner sep=-5pt, minimum width=0.35cm, string, fill=white}}
\tikzset{dimension/.style={font=\scriptsize, inner sep=1pt}}

\tikzset{arrow data/.style 2 args={
      decoration={
         markings,
         mark=at position #1 with \arrow{#2}},
         postaction=decorate}
}

\tikzset{along path/.style={every path/.style={}, sloped, allow upside down}}

\def\zxnormal {
                \def \zxscale{0.55}
                \def\zxnodescale{0.8}
                \def\vertexscale{0.7}
                \def\zxshift{0.075cm}
                \def\hadscale{0.8}
                \def\trianglescale{1}
                \def\boxscale{1}
                }

\def\zxgreen{white}

\def\zxwhite{white}

\tikzset{front/.style ={node on layer=foreground}}
\tikzset{zx/.style = {string, scale=\zxscale}}
\tikzset{zxnode/.style n args={1}{blob,scale=\zxnodescale,fill=#1,node on layer=foreground}}
\tikzset{box/.style={draw, rectangle, fill=white, inner sep=1pt, minimum width=10pt,minimum height=10pt, font=\scriptsize, line width=0.7pt,scale=\zxnodescale,node on layer=foreground}}
\tikzset{boxvertex/.style={draw, rectangle, fill=white, line width=0.733pt,scale=0.75*\vertexscale}}
\tikzset{bigbox/.style={draw, rectangle, fill=white,  minimum width=\boxscale *18pt,minimum height=\boxscale*8pt, line width=0.7pt,scale=\zxnodescale}}
\newlength{\unitbox}
\tikzset{widebox/.style ={draw,rectangle, fill=white, line width=0.7pt,scale=0.75*\zxnodescale,minimum height=15pt,inner sep=1pt,  minimum width = \unitbox,   anchor=center }}
\tikzset{wideboxm/.style n args={1}{draw,rectangle, fill=white, line width=0.7pt,scale=0.75*\zxnodescale,minimum height=15pt,inner sep=1pt,  minimum width =2\unitbox+#1\unitbox,   anchor=center }}
\tikzset{triangleup/.style n args={1}{draw, shape=isosceles triangle, isosceles triangle stretches, fill=white, line width=0.7pt,scale=0.75*\zxnodescale,minimum height=15pt,inner sep=1pt,  minimum width = #1*\trianglescale cm +0.15*\trianglescale cm,  shape border rotate=90, anchor=south }}
\tikzset{triangledown/.style n args={1}{draw, shape=isosceles triangle, isosceles triangle stretches, fill=white, line width=0.7pt,scale=0.75*\zxnodescale,minimum height=15pt,inner sep=1pt,  minimum width = #1*\trianglescale cm +0.15*\trianglescale cm,  shape border rotate=-90, anchor=north }}
\tikzset{zxvertex/.style n args={1}{draw,fill=#1,circle,line width=0.7pt,scale=0.75*\vertexscale}}
\tikzset{zxdown/.style={yshift=-\zxshift}}
\tikzset{zxup/.style={yshift=\zxshift}}

\newcommand\mult[3]{ 
\draw[string] (#1.center) to [out=up, in=-135] +(0.5*#2,#3) to [out=-45, in=up] +(0.5*#2,-#3);
\node[zxvertex=\zxgreen,zxdown] at ($(#1)+(0.5*#2,#3)$){};
}
\newcommand\comult[3]{ 
\draw[string] (#1.center) to [out=down, in=135] +(0.5*#2,-#3) to [out=45, in=down] +(0.5*#2,#3);
\node[zxvertex=\zxgreen,zxup] at ($(#1) +(0.5*#2,-#3)$){};}

\newcommand\unit[2]{ 
\draw[string] (#1.center) to + (0, -#2);
\node[zxvertex=\zxgreen] at ($(#1) +(0,-#2)$){};
}

\newcommand{\Tr}{\mathrm{Tr}}

\renewcommand{\to}[1][]{\ensuremath{\xrightarrow{#1}}}

\allowdisplaybreaks[1]

\theoremstyle{plain} 

\newtheorem{theorem}{Theorem}[section]
\newtheorem{lemma}[theorem]{Lemma}
\newtheorem{corollary}[theorem]{Corollary}          
\newtheorem{proposition}[theorem]{Proposition}

\theoremstyle{definition} 
\newtheorem{definition}[theorem]{Definition}

\newtheorem{remark}[theorem]{Remark}

\theoremstyle{remark}  

\newtheoremstyle{special_statement} 
        {\topskip}
        {\topskip}
        {\addtolength{\leftskip}{2.5em} \itshape }
        {}
        {\bfseries}
        {:}
        {.5em}
        {}
\theoremstyle{special_statement}

\DeclareMathOperator{\Hom}{Hom}
\DeclareMathOperator{\End}{End}

\DeclareMathOperator{\Fun}{Fun}
\newcommand{\id}{\mathrm{id}}

\newcommand{\Rep}{\mathrm{Rep}}

\newcommand{\Hilb}{\ensuremath{\mathrm{Hilb}}}

\newcommand{\Obj}{\ensuremath{\mathrm{Obj}}}

\makeatletter
\DeclareFontFamily{OMX}{MnSymbolE}{}
\DeclareSymbolFont{MnLargeSymbols}{OMX}{MnSymbolE}{m}{n}
\SetSymbolFont{MnLargeSymbols}{bold}{OMX}{MnSymbolE}{b}{n}
\DeclareFontShape{OMX}{MnSymbolE}{m}{n}{
    <-6>  MnSymbolE5
   <6-7>  MnSymbolE6
   <7-8>  MnSymbolE7
   <8-9>  MnSymbolE8
   <9-10> MnSymbolE9
  <10-12> MnSymbolE10
  <12->   MnSymbolE12
}{}
\DeclareFontShape{OMX}{MnSymbolE}{b}{n}{
    <-6>  MnSymbolE-Bold5
   <6-7>  MnSymbolE-Bold6
   <7-8>  MnSymbolE-Bold7
   <8-9>  MnSymbolE-Bold8
   <9-10> MnSymbolE-Bold9
  <10-12> MnSymbolE-Bold10
  <12->   MnSymbolE-Bold12
}{}

\let\llangle\@undefined
\let\rrangle\@undefined
\DeclareMathDelimiter{\llangle}{\mathopen}%
                     {MnLargeSymbols}{'164}{MnLargeSymbols}{'164}
\DeclareMathDelimiter{\rrangle}{\mathclose}%
                     {MnLargeSymbols}{'171}{MnLargeSymbols}{'171}
\makeatother

\usepackage[colorlinks=true, linkcolor=black, citecolor=black, urlcolor=black, hypertexnames=false
        ]{hyperref}

\newcounter{DRcomment}
\newcounter{DVcomment}
\newcounter{BMcomment}
\newcounter{JVcomment}
\newcommand\ignore[1]{}

\tikzstyle{blackdot}=[circle, draw=black, fill=black, inner sep=.5ex, line width=\thickness, node on layer=foreground]
\tikzstyle{whitedot}=[circle, draw=black, fill=white, inner sep=.5ex, line width=\thickness, node on layer=foreground]

\tikzset{proofdiagram/.style={scale=1}}
\newlength\morphismheight
\newlength\minimummorphismwidth
\newlength\stateheight
\usepackage{diagrams}
\usepackage[utf8]{inputenc}

\title{Unitary pseudonatural transformations}
\author{Dominic Verdon}
\date{School of Mathematics \\ University of Bristol \\[2ex]
\today}

\begin{document}

\normalsize
\zxnormal
\maketitle      
\abstract{We suggest two approaches to a definition of unitarity for pseudonatural transformations between unitary pseudofunctors on pivotal dagger 2-categories. The first is to require that the 2-morphism components of the transformation be unitary. The second is to require that the dagger of the transformation be equal to its inverse. We show that the `inverse' making these definitions equivalent is the right dual of the transformation in the 2-category $\Fun(\mathcal{C},\mathcal{D})$ of pseudofunctors $\mathcal{C} \to \mathcal{D}$, pseudonatural transformations, and modifications. We show that the subcategory $\Fun_u(\mathcal{C},\mathcal{D}) \subset \Fun(\mathcal{C},\mathcal{D})$ whose objects are unitary pseudofunctors and whose 1-morphisms are unitary pseudonatural transformations is a pivotal dagger 2-category. We apply these results to obtain a Morita-theoretical classification of unitary pseudonatural transformations between fibre functors on the category of representations of a compact quantum group.}

\section{Introduction}

\subsection{Overview} Natural transformations between functors are a crucial element of category theory. Let $\mathcal{C},\mathcal{D}$ be categories and $F,F': \mathcal{C} \to \mathcal{D}$ be functors. We say that a natural transformation $\alpha: F \to F'$ is \emph{invertible} if its components $\{\alpha_X\}_{X \in \Obj(\mathcal{C})}$ are invertible in $\mathcal{D}$. If $\mathcal{D}$ is a dagger category, then we say that an invertible natural transformation is \emph{unitary} if its components are additionally unitary in $\mathcal{D}$.

Perhaps more naturally, these notions of invertibility may be defined with respect to the category $\Fun(\mathcal{C},\mathcal{D})$ of functors and natural transformations. An invertible natural transformation is just an invertible morphism in this category. If $\mathcal{C}, \mathcal{D}$ are dagger categories, the subcategory of $\Fun(\mathcal{C},\mathcal{D})$ whose objects are unitary functors inherits a dagger structure; a unitary natural transformation is a unitary morphism in this dagger category. 

Just as natural transformations between functors are an important part of category theory, pseudonatural transformations between pseudofunctors are an important part of 2-category theory, which includes monoidal category theory. In this work we consider the generalisation of the aforementioned notions of invertibility to pseudonatural transformations.\footnote{We remark that our results about duality generalise straightforwardly to oplax natural transformations, although for applications we did not require this level of generality.}

Let $\mathcal{C},\mathcal{D}$ be 2-categories, and let $\Fun(\mathcal{C},\mathcal{D})$ be the 2-category of pseudofunctors $\mathcal{C} \to \mathcal{D}$, pseudonatural transformations and modifications. We consider invertibility of a pseudonatural transformation as a 1-morphism in $\Fun(\mathcal{C},\mathcal{D})$.

We could consider equivalences in $\Fun(\mathcal{C},\mathcal{D})$. However, we find that this notion of invertibility is too strong for our purposes. A weaker notion of invertibility of a 1-morphism in a 2-category is duality, a.k.a. adjunction. A 2-category is said to `have right (resp. left) duals' when every 1-morphism has a chosen right (resp. left) dual. A coherent choice of left and right duals for every object is called a  pivotal structure; a 2-category with a pivotal structure is called \emph{pivotal}.
Here we unpack the notion of duality for pseudonatural transformations (Definition~\ref{def:dualpnt}) and show the following facts.
\begin{itemize}
\item If $\mathcal{C}$ has left (resp. right) duals and $\mathcal{D}$ has right (resp. left) duals, then $\Fun(\mathcal{C},\mathcal{D})$ has right (resp. left) duals (Corollary~\ref{cor:dualsexistence}).
\item If $\mathcal{C}, \mathcal{D}$ are pivotal, then $\Fun_{p}(\mathcal{C},\mathcal{D})$ is also pivotal, where the subscript $p$ represents restriction to pivotal functors. (Proposition~\ref{prop:pivinduced}).
\end{itemize}
If the 2-categories $\mathcal{C},\mathcal{D}$ additionally have a dagger structure, we restrict $\Fun(\mathcal{C},\mathcal{D})$ to unitary pseudofunctors. We now consider the notion of \emph{unitarity} of a pseudonatural transformation. This consideration is motivated either physically, by the desire that the components of the transformation should be unitary in $\mathcal{D}$; or categorically, by the desire that the 2-category $\Fun(\mathcal{C},\mathcal{D})$ should itself inherit a dagger structure (for general pseudonatural transformations, there is no obvious dagger structure on $\Fun(\mathcal{C},\mathcal{D})$). 

We could say that a pseudonatural transformation is unitary when all its 2-morphism components are unitary in $\mathcal{D}$. This is our first definition of a \emph{unitary pseudonatural transformation}. However, the more categorically natural way of specifying unitarity of a pseudonatural transformation is to say that its \emph{dagger} is equal to its inverse (i.e. its right dual). When $\mathcal{C},\mathcal{D}$ are pivotal dagger (i.e. possessing compatible pivotal and dagger structures), we observe that there is a notion of the dagger of a pseudonatural transformation such that the following definitions of a unitary pseudonatural transformation are equivalent (Lemma~\ref{lem:unitaritydefsequiv}):
\begin{itemize} 
\item A pseudonatural transformation all of whose 2-morphism components are unitary. 
\item A pseudonatural transformation whose right dual is equal to its dagger.
\end{itemize}
Let $\Fun_u(\mathcal{C},\mathcal{D}) \subset \Fun(\mathcal{C},\mathcal{D})$ be the subcategory whose objects are unitary pseudofunctors and whose 1-morphisms are unitary pseudonatural transformations. The category $\Fun_u(\mathcal{C},\mathcal{D})$ inherits a dagger structure from $\mathcal{D}$. Moreover, pivotality comes `for free', with no need to restrict to pivotal functors.
\begin{itemize}
\item Let $\mathcal{C},\mathcal{D}$ be pivotal dagger categories. Then the category $\Fun_u(\mathcal{C},\mathcal{D})$ is a pivotal dagger category. (Theorem~\ref{thm:funcddagger}.)
\end{itemize}
Our main motivation for this work is the study of unitary pseudonatural transformations between fibre functors on representation categories of compact quantum groups, which are the subject of the paper~\cite{Verdon2020}. In particular, the results in this paper allow us to classify fibre functors and unitary pseudonatural transformations between them in terms of Morita theory in the 2-category $\Fun_u(\Rep(G),\Hilb)$. To show this we prove a more general result (Theorem~\ref{thm:starisoclass}) which relates equivalence classes of 1-morphisms out of an object in a pivotal dagger 2-category to unitary $*$-isomorphism classes of special Frobenius algebras in its endomorphism category.

\subsection{Acknowledgements}
The author thanks Ashley Montanaro, David Reutter, Changpeng Shao and Jamie Vicary for useful discussions related to this work, and is grateful to anonymous referees for many useful comments. The work was supported by EPSRC.

\subsection{Structure}
In Section~\ref{sec:background} we introduce necessary background material for the rest of this paper. In Section~\ref{sec:pnts} we recall the basic theory of pseudonatural transformations. In Section~\ref{sec:duals} we discuss dualisability of pseudonatural transformations. In Section~\ref{sec:unitary} we consider unitary pseudonatural transformations. In Section~\ref{sec:apps} we consider an application of our results to the study of fibre functors on representation categories of compact quantum groups.

\section{Background: Pivotal dagger 2-categories}\label{sec:background}

\subsection{Diagrams for 2-categories}
\ignore{A 2-category is a generalisation of a  category. While a category has objects, morphisms, and composition laws, a 2-category has objects, morphisms, and morphisms between the morphisms, called 2-morphisms, obeying composition laws. The general `weak' definition of 2-category can be found in e.g.~\cite{Leinster1998}. Roughly, a 2-category $\mathcal{C}$ is defined by a set of objects of objects $r,s,\dots$, together with a category of morphisms $\mathcal{C}(r,s)$ for every pair of objects, and functors $\mathcal{C}(r,s) \times \mathcal{C}(s,t) \to \mathcal{C}(r,t)$ defining composition of these Hom-categories, with various coherence data.
2-categories are much more manageable than the general definition might suggest.}
Recall that every monoidal category is equivalent to a strict monoidal category~\cite{MacLane1963}. This allows us to assume our monoidal categories are strict, allowing the use of a convenient and well-known diagrammatic calculus~\cite{Selinger2010}.
In 2-category theory, a similar strictification result holds --- every weak 2-category is equivalent to a strict 2-category~\cite{Leinster1998}. We can therefore also use a diagrammatic calculus in this case. 

A monoidal category is precisely a 2-category with a single object, where 1-morphisms are the `objects' of the monoidal category, 2-morphisms are the `morphisms', and composition of 1-morphisms is the `monoidal product'. The 2-categorical diagrammatic calculus is nothing more than the diagrammatic calculus for monoidal categories enhanced with region labels. We briefly summarise this calculus now, closely following the exposition in~\cite{Marsden2014}. More information can be found in e.g.~\cite{Hummon2012}.

Objects $r,s, \cdots$ of a 2-category are represented by labelled two-dimensional regions of a planar diagram:
\begin{calign}\nonumber
\includegraphics{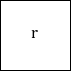}
\end{calign}
The 1-morphisms $X: r \to s$ are represented by edges, separating the region $r$ on the left from the region $s$ on the right:
\begin{calign}\nonumber
\includegraphics{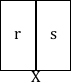}
\end{calign}
Edges corresponding to identity 1-morphisms $\id_r: r \to r$ are invisible in the diagrammatic calculus. 

The 1-morphisms compose from left to right. That is, for 1-morphisms $X:r \to s, Y: s \to t$, the composite $X \otimes Y: r \to t$ is represented as:
\begin{calign}\nonumber
\includegraphics{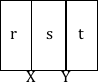}
\end{calign}
For two parallel 1-morphisms $X, Y: r \to s$, a  2-morphism $\alpha: X \to Y$ is represented by a vertex in the diagram, drawn as a box:
\begin{calign}\nonumber
\includegraphics{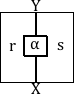}
\end{calign}
The 2-morphisms can compose in two ways, depending on their type. For parallel 1-morphisms $X,Y,Z: r \to s$, 2-morphisms $\alpha:X \to Y, \beta: Y \to Z$ can be composed `vertically' to obtain a 2-morphism $\beta \circ \alpha: X \to Z$. This is represented by vertical juxtaposition in the diagram:
\begin{calign}\nonumber
\includegraphics{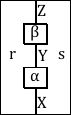}
\end{calign}
For 1-morphisms $X,X': r \to S$ and $Y,Y': s \to t$, 2-morphisms $\alpha: X \to X'$ and $\beta: Y \to Y'$ can be composed `horizontally' to obtain a 2-morphism $\alpha \otimes \beta: X \otimes Y \to X' \otimes Y'$. This is represented by horizontal juxtaposition in the diagram:
\begin{calign}\nonumber
\includegraphics{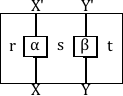}
\end{calign}
As with 1-morphisms, the identity 2-morphisms $\id_X: X \to X$ are invisible in the diagrammatic calculus.

All 2-categories satisfy the \emph{interchange law}. For any 1-morphisms $X,X',X'': r \to s$ and $Y,Y',Y'': s \to t$, and 2-morphisms $\alpha: X \to X'$, $\alpha': X' \to X''$, $\beta: Y \to Y',\beta':Y' \to Y''$:
$$
(\alpha' \circ \alpha) \otimes (\beta' \circ \beta) = (\alpha' \otimes \beta') \circ (\alpha \otimes \beta)
$$
This corresponds to the following diagram having an unambiguous interpretation as a 2-morphism:   
\begin{align}\label{eq:interchange}
\includegraphics[valign=c]{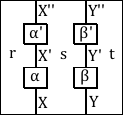}
\end{align}
We also have the following \emph{sliding equalities}, which may be obtained by taking some morphisms to be the identity in~\eqref{eq:interchange}:
\begin{align*}\nonumber
\includegraphics[valign=c]{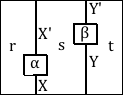}
~~=~~
\includegraphics[valign=c]{Figures/svg/2cats/2morphismhcomp.pdf}
~~=~~
\includegraphics[valign=c]{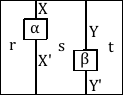}
\end{align*}
These equalities allow us to move 2-morphism boxes past each other provided there are no obstructions. 

Before moving onto pseudofunctors, we give a first definition from 2-category theory. Equivalence is a strong notion of invertibility of a 1-morphism in a 2-category. From now on we will not draw an enclosing box around diagrams. 
\begin{definition}\label{def:equiv}
Let $\mathcal{C}$ be a 2-category and let $X: r \to s$ be a 1-morphism in $\mathcal{C}$. We say that $X$ is an \emph{equivalence} if there exists a 1-morphism $X^{-1}: s \to r$, and invertible\footnote{I.e. invertible in the Hom-categories $\mathcal{C}(r,r)$ and $\mathcal{C}(s,s)$. We sometimes call an invertible 2-morphism a \emph{2-isomorphism}.} 2-morphisms $\alpha: \id_r \to X \otimes X^{-1}$ and $\beta: \id_s \to X^{-1} \otimes X$. In diagrams, the equations for invertibility of $\alpha, \beta$ are as follows, where $\alpha^{-1},\beta^{-1}$ are the inverse 2-morphisms:
\begin{align}\label{eq:equivalence}
\includegraphics[valign=c]{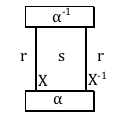}
~~=
\includegraphics[valign=c]{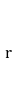},
&&
\includegraphics[valign=c]{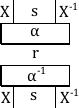}
=
\includegraphics[valign=c]{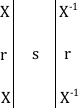},
&&
\includegraphics[valign=c]{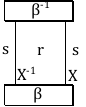}
=
\includegraphics[valign=c]{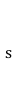},
&&
\includegraphics[valign=c]{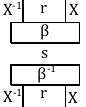}
=
\includegraphics[valign=c]{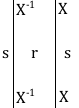}
\end{align}
If there exists an equivalence $X: r \to  s$ we say that the objects $r$ and $s$ are \emph{equivalent} in $\mathcal{C}$.
\end{definition}
\subsection{Diagrams for pseudofunctors}
While our 2-categories are strictified, allowing us to use the diagrammatic calculus, we will consider functors between them which are not strict. For this, we use a graphical calculus of \emph{functorial boxes} previously applied in the special case of  monoidal functors~\cite{Mellies2006}.
\begin{definition}
Let $\mathcal{C}, \mathcal{D}$, be 2-categories. A \emph{pseudofunctor} $F: \mathcal{C} \to \mathcal{D}$ consists of the following data.
\begin{itemize}
\item For each object $r$ of $\mathcal{C}$, an object $F(r)$ of $\mathcal{D}$.
\item For each hom-category $\mathcal{C}(r,s)$ of $\mathcal{C}$, a functor $F_{r,s}: \mathcal{C}(r,s) \to \mathcal{D}(F(r),F(s))$.

In the graphical calculus, we represent the effect of the functor $F_{r,s}$ by drawing a shaded box around 1- and 2-morphisms in $\mathcal{C}(r,s)$. For example, $X,Y: r \to s$ be 1-morphisms and  $f: X \to Y$ a 2-morphism in $\mathcal{C}$. Then the 2-morphism $F(f): F(X) \to F(Y)$ in $\mathcal{D}(F(r), F(s))$ is represented as:
\begin{calign}\nonumber
\includegraphics[scale=0.8,valign=c]{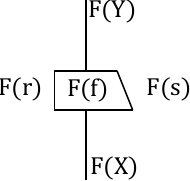}
~~=~~
\includegraphics[scale=0.8,valign=c]{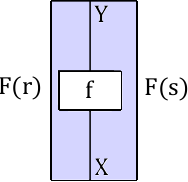}
\end{calign}
\item For every pair of composable 1-morphisms $X:r \to s$, $Y: s \to t$ of $\mathcal{C}$, an invertible \emph{multiplicator} 2-morphism $m_{X,Y}:F(X) \otimes_{\mathcal{D}} F(Y) \to F(X \otimes_{\mathcal{C}} Y)$. In the graphical calculus, these 2-morphisms and their inverses are represented as follows:
\begin{calign}\label{eq:multiplicator}
\includegraphics[scale=.8,valign=c]{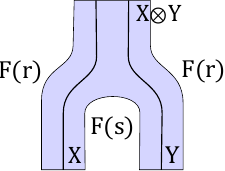}
&
\includegraphics[scale=.7,valign=c]{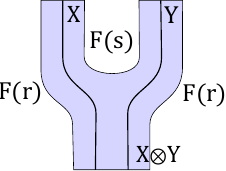}
\\\nonumber
m_{X,Y}: F(X) \otimes_{\mathcal{D}}  F(Y) \to F(X \otimes_{\mathcal{C}} Y) & m_{X,Y}^{-1}: F(X \otimes_{\mathcal{C}} Y) \to F(X) \otimes_{\mathcal{D}}  F(Y)
\end{calign}
\item For every object $r$ of $\mathcal{C}$, an invertible `unitor' 2-morphism $u_r: \id_{F(r)} \to F(\id_{r})$. In the diagrammatic calculus, these 2-morphism and their inverses are represented as follows (recall that identity 1-morphisms are invisible):
\begin{calign}\label{eq:unitor}
\includegraphics[scale=.8,valign=c]{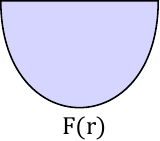}
&
\includegraphics[scale=.8,valign=c]{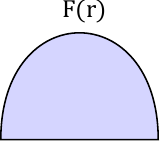} \\\nonumber
u_r: \id_{F(r)} \to F(\id_{r}) & u_r^{-1}: F(\id_{r})  \to \id_{F(r)}
\end{calign}
\end{itemize}
The multiplicators and unitor must obey the following coherence equations:
\begin{itemize}
\item \emph{Naturality}. For any objects $r,s,t$, 1-morphisms $X,X': r \to s$, $Y,Y': s \to t$, and 2-morphisms $f:X \to X', g: Y \to Y'$ in $\mathcal{C}$:
\begin{calign}\label{eq:psfctnat}
\includegraphics[scale=0.8,valign=c]{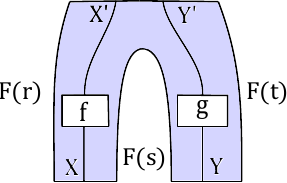}
~~=~~
\includegraphics[scale=0.8,valign=c]{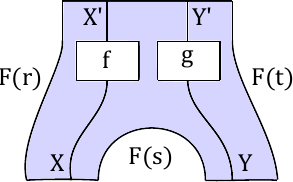}
\end{calign}
\item \emph{Associativity}. For any objects $r,s,t,u$ and 1-morphisms $X:r \to s$, $Y: s \to t$, $Z: t \to u$ of $\mathcal{C}$:
\begin{calign}\label{eq:psfctassoc}
\includegraphics[scale=0.8,valign=c]{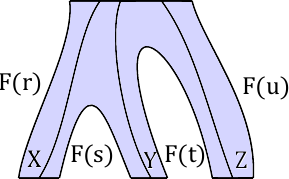}
~~=~~
\includegraphics[scale=0.8,valign=c]{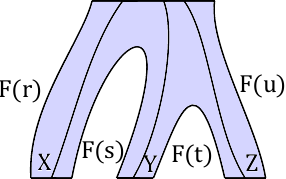}
\end{calign}
\item \emph{Unitality}. For any objects $r,s$ and 1-morphism $X:r \to s$ of $\mathcal{C}$:
\begin{calign}\label{eq:psfctunital}
\includegraphics[scale=0.8,valign=c]{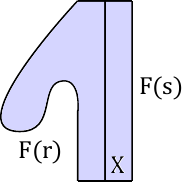}
~~=~~
\includegraphics[scale=0.8,valign=c]{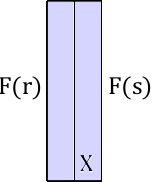}
~~=~~
\includegraphics[scale=0.8,valign=c]{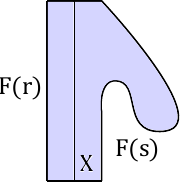}
\end{calign}
\end{itemize}
We say that a pseudofunctor $F: \mathcal{C} \to \mathcal{D}$ is an \emph{equivalence} if every object in $\mathcal{D}$ is equivalent to an object in the image of $F$ (Definition~\ref{def:equiv}) and the functors $F_{r,s}: \mathcal{C}(r,s) \to \mathcal{D}(r,s)$ are equivalences.
\end{definition}
\noindent
We remark that the analogous \emph{conaturality}, \emph{coassociativity} and \emph{counitality} equations for the inverses $\{m_{X,Y}^{-1}\},\{u_r^{-1}\}$, obtained by reflecting~(\ref{eq:psfctnat}-\ref{eq:psfctunital}) in a horizontal axis, are already implied by~(\ref{eq:psfctnat}-\ref{eq:psfctunital}). To give some idea of the calculus of functorial boxes, we explicitly prove the following lemma and proposition. From now on we will unclutter the diagrams by omitting region and 1-morphism labels, unless adding the labels seems to significantly aid comprehension.
\begin{lemma}\label{lem:pushpast}
For any objects $r,s,t,u$ and 1-morphisms $X: r \to s$, $Y: s \to t$, $Z: t \to u$, the following equations are satisfied:
\begin{calign}\nonumber
\includegraphics[scale=1,valign=c]{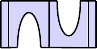}
~~=~~
\includegraphics[scale=1,valign=c]{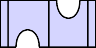}
&
\includegraphics[scale=1,valign=c]{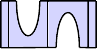}
~~=~~
\includegraphics[scale=1,valign=c]{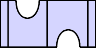}
\end{calign}
\end{lemma}
\begin{proof}
We prove the left equation; the right equation is proved similarly. 
\begin{calign}\nonumber
\includegraphics[scale=1,valign=c]{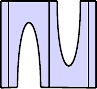}
~~=~~
\includegraphics[scale=1,valign=c]{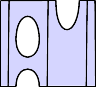}
~~=~~
\includegraphics[scale=1,valign=c]{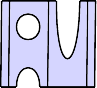}
~~=~~
\includegraphics[scale=1,valign=c]{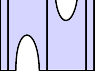}
\end{calign}
Here the first and third equalities are by invertibility of $m_{X,Y}$, and the second is by coassociativity.
\end{proof}
\noindent
With Lemma~\ref{lem:pushpast}, the equations~(\ref{eq:psfctnat}-\ref{eq:psfctunital}) are sufficient to deform functorial boxes topologically as required.  From now on we will do this mostly without comment.

\subsection{Pseudonatural transformations}\label{sec:pnts}
We now recall the definition of a pseudonatural transformation between pseudofunctors~\cite{Leinster1998}.
\begin{definition}\label{def:pnt}
Let $\mathcal{C},\mathcal{D}$ be 2-categories, and let $F,G: \mathcal{C} \to \mathcal{D}$ be pseudofunctors (depicted by blue and red boxes respectively). A \emph{pseudonatural transformation} $\alpha: F \to G$ is defined by the following data:
\begin{itemize}
\item For every object $r$ of $\mathcal{C}$, a 1-morphism $\alpha_r:F(r) \to G(r)$ of $\mathcal{D}$ (drawn as a green wire). 
\item For every 1-morphism $X:r \to s$ of $\mathcal{C}$, an invertible 2-morphism $\alpha_X: F(X) \otimes \alpha_s \to \alpha_r \otimes G(X)$ (drawn as a white vertex):
\begin{calign}\label{eq:pntdef}
\includegraphics[valign=c]{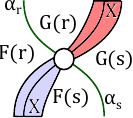}
\end{calign}
\end{itemize}
The 1-morphisms $\alpha_X$ must satisfy the following conditions:
\begin{itemize}
\item \emph{Naturality}. For every 2-morphism $f: X \to Y$ in $\mathcal{C}$: 
\begin{calign}\label{eq:pntnat}
\includegraphics[valign=c]{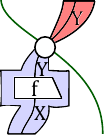}
~~=~~
\includegraphics[valign=c]{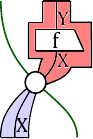}
\end{calign}
\item \emph{Monoidality.} \begin{itemize}\item For every pair of 1-morphisms $X: r \to s, Y: s \to t$ in $\mathcal{C}$:
\begin{calign}\label{eq:pntmon}
\includegraphics[scale=.9,valign=c]{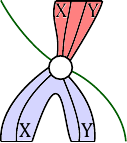}
~~=~~
\includegraphics[scale=.9,valign=c]{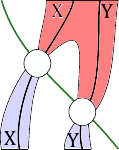}
\end{calign}
\item For every object $r$ of $\mathcal{C}$:
\begin{calign}\label{eq:pntmonunit}
\includegraphics[scale=1,valign=c]{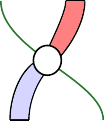}
~~=~~
\includegraphics[scale=1.25,valign=c]{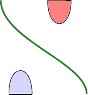}
\end{calign}
\end{itemize}
(Equation~\eqref{eq:pntmon} already implies the analogous pullthroughs for the comultiplicators $\{m_{X,Y}^{-1}\}$.)
\end{itemize}
If $\alpha_r = \id_{F(r)}$ for every object $r$ of $\mathcal{C}$, we say that $\alpha$ is an \emph{invertible icon}~\cite{Lack2010}.  In particular, if $\mathcal{C}, \mathcal{D}$ are one-object 2-categories and $\alpha$ is an invertible icon, we recover the standard notion of monoidal natural isomorphism. 
\end{definition}
\begin{remark}
The results we will prove in Section~\ref{sec:duals} extend to oplax natural transformations (i.e. where the 2-morphism components are not invertible). However, we did not need this level of generality for applications.
\end{remark}
\ignore{
\begin{remark}
There is an unavoidable choice of convention in the definition of a pseudonatural transformation regarding whether the $A$-wire should go from left to right or from right to left.
\end{remark}}
\noindent
Pseudonatural transformations $\alpha: F \to G$ and $\beta: G \to H$ can be composed associatively. We define $\alpha \otimes \beta: F \to H$ as follows.
\begin{itemize}
\item For every object  $r$ of $\mathcal{C}$, $(\alpha \otimes \beta)_r := \alpha_r \otimes \beta_r$. 
\item For any 1-morphism $X: r \to s$ of $\mathcal{C}$, $(\alpha \otimes \beta)_X$ is defined as the following composite (we colour the $\beta$-wire orange, and the $H$-box brown):
\begin{calign}
\includegraphics[valign=c]{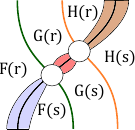}
\end{calign}
\end{itemize}
There are also morphisms between pseudonatural transformations, known as \emph{modifications}~\cite{Leinster1998}.
\begin{definition}
Let $\alpha, \beta: F \Rightarrow G$ be pseudonatural transformations between pseudofunctors $F,G: \mathcal{C} \to \mathcal{D}$. (We colour the $\alpha$-wire green and the $\beta$-wire orange.) A \emph{modification} $f: \alpha \to \beta$ is defined by the following data:
\begin{itemize}
\item For every object $r$ of $\mathcal{C}$, a 2-morphism $f_r: \alpha_r \to \beta_r$ in $\mathcal{D}$, such that the 2-morphisms $\{f_r\}$ satisfy the following equation for all 1-morphisms $X:r \to s$ in $\mathcal{C}$:
\begin{calign}
\includegraphics[valign=c]{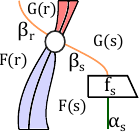}
~~=~~
\includegraphics[valign=c]{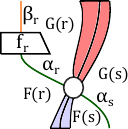}
\end{calign}
\end{itemize}
\end{definition} 
\noindent
Modifications can themselves be composed horizontally and vertically in an obvious way. Altogether, this compositional structure is again a 2-category.
\begin{definition}
Let $\mathcal{C},\mathcal{D}$ be 2-categories. The 2-category $\Fun(\mathcal{C},\mathcal{D})$ is defined as follows:
\begin{itemize}
\item Objects: pseudofunctors $F,G,\dots,\cdot: \mathcal{C} \to \mathcal{D}$.
\item 1-morphisms: pseudonatural transformations $\alpha,\beta,\dots: F \to G$.
\item 2-morphisms: modifications $f, g,\dots: \alpha \to \beta$.
\end{itemize}
As we are assuming that $\mathcal{C}$ and $\mathcal{D}$ are strict, strictness of $\Fun(\mathcal{C},\mathcal{D})$ follows.
\ignore{
By Definition~\ref{}, the vertical composition $\mu \circ \lambda$ of modifications $\lambda: (\alpha,A) \to (\beta, B)$ and $\mu: (\beta,B) \to (\delta, D)$ is
$$DRAW$$ 
while the horizontal composition of modifications $\lambda: (\alpha,A) \to (\beta, B)$
}
\end{definition}
\noindent

\subsection{Pivotal 2-categories}

In a 2-category the most general notion of invertibility of a 1-morphism is \emph{duality}, also known as \emph{adjunction}.
\begin{definition}\label{def:duals}
Let $X: r \to s$ be a 1-morphism in a 2-category. A \emph{right dual} $[X^*,\eta,\epsilon]$ for $X$ is:
\begin{itemize}
\item A 1-morphism $X^*: s \to r$.
\item Two 2-morphisms $\eta: \id_s \to X^* \otimes X$ and $\epsilon: X \otimes X^* \to \id_r$ satisfying the following \emph{snake equations}: 
\begin{calign}\label{eq:rightsnakes}
\includegraphics[valign=c]{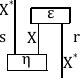}
~~=~~
\includegraphics[valign=c]{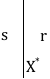}
&
\includegraphics[valign=c]{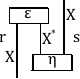}
~~=~~
\includegraphics[valign=c]{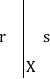}
\end{calign}
\end{itemize}
A \emph{left dual} $[{}^*X,\eta,\epsilon]$ is defined similarly, with 2-morphisms $\eta: \id_s \to X \otimes {}^*X$ and $\epsilon: {}^*X \otimes X \to \id_r$ satisfying the analogues of~\eqref{eq:rightsnakes}. 

We say that a 2-category $\mathcal{C}$ \emph{has right duals} (resp. \emph{has left duals}) if every 1-morphism $X$ in $\mathcal{C}$ has a chosen right dual $[X^*,\eta,\epsilon]$ (resp. a chosen left dual).
\end{definition}
\noindent
To represent duals in the graphical calculus, we label the $X$-wire and the $X^*$-wire with the label $X$, draw an upwards-pointing arrow on the $X$-wire and a downward-pointing arrow on the $X^*$-wire, and write $\eta$ and $\epsilon$ as a cup and a cap, respectively. Then the equations~\eqref{eq:rightsnakes} appear as follows:
\begin{calign}\nonumber
\includegraphics[valign=c]{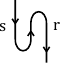}
~~=~~
\includegraphics[valign=c]{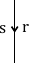}
~~~~~~~~
\includegraphics[valign=c]{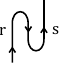}
~~=~~
\includegraphics[valign=c]{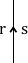}
&&
\includegraphics[valign=c]{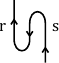}
~~=~~
\includegraphics[valign=c]{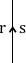}
~~~~~~~~
\includegraphics[valign=c]{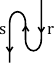}
~~=~~
\includegraphics[valign=c]{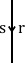}\\\nonumber
\text{right dual} &&\text{left dual}
\end{calign}
Since the graphical calculus for 2-categories is just a `region-labelled' version of the graphical calculus for monoidal categories, various statements about duals in monoidal categories  immediately generalise to duals in 2-categories. We recall some of these statements now.
\begin{lemma}[{\cite[Lemmas 3.6, 3.7]{Heunen2019}}]\label{lem:nestedduals}
If $[X^*,\eta_X,\epsilon_X]$ and $[Y^*,\eta_Y,\epsilon_Y]$ are right duals for $X:r \to s$ and $Y: s \to t$ respectively, then $[Y^* \otimes X^*, \eta_{X \otimes Y},\epsilon_{X \otimes Y}]$ is right dual to $X \otimes Y$, where $\eta_{X \otimes Y}$ and $\epsilon_{X \otimes Y}$ are defined by:
\begin{calign}\label{eq:nestedduals}
\includegraphics[valign=c]{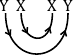}
&&
\includegraphics[valign=c]{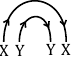}
\\\nonumber
\eta_{X \otimes Y} && \epsilon_{X \otimes Y}
\end{calign}
Moreover, for any object $r$, $[\id_r,\id_{\id_r},\id_{\id_r}]$ is right dual to $\id_r$. Analogous statements hold for left duals.
\end{lemma}
\begin{lemma}[{\cite[Lemma 3.4]{Heunen2019}}]\label{lem:relateduals}
Let $X: r \to s$ be a 1-morphism, and let $[X^*,\eta,\epsilon],[X^*{}',\eta',\epsilon']$ be right duals. Then there is a unique 2-isomorphism $\alpha: X^* \to X^*{}'$ such that 
\begin{calign}\label{eq:relateduals}
\includegraphics[valign=c]{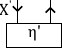}
~~=~~
\includegraphics[valign=c]{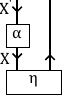}
&
\includegraphics[valign=c]{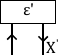}
~~=~~
\includegraphics[valign=c]{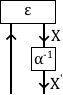}
\end{calign}
An analogous statement holds for left duals.
\end{lemma}
\noindent
In a 2-category with duals, we can define a notion of transposition for 2-morphisms.
\begin{definition}
Let $X,Y: r \to s$ be 1-morphisms with chosen right duals $[X^*,\eta_X,\epsilon_X]$ and $[Y^*,\eta_Y,\epsilon_Y]$. For any 2-morphism $f:X \to Y$, we define its \emph{right transpose} (a.k.a. \emph{mate}) $f^*: Y^* \to X^*$ as follows:
\begin{calign}\label{eq:rtranspose}
\includegraphics[valign=c]{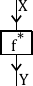}
~~=~~
\includegraphics[valign=c]{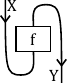}
\end{calign}
For left duals ${}^*X,{}^*Y$, a \emph{left transpose} may be defined analogously.
\end{definition}
\noindent
In this work we are mostly interested in categories with compatible left and right duals. Such categories are called \emph{pivotal}. Let $\mathcal{C}$ be a 2-category with right duals. It is straightforward to check that the following defines an identity-on-objects pseudofunctor $\mathcal{C} \to \mathcal{C}$, which we call the \emph{double duals} pseudofunctor:
\begin{itemize} 
\item 1-morphisms $X:r \to s$ are taken to the double dual $X^{**}:=(X^*)^*$.
\item 2-morphisms $f: X \to Y$ are taken to the double transpose $f^{**}:=(f^*)^*$.
\item The multiplicators $m_{X,Y}$ and unitors $u_r$ are defined using the isomorphisms of Lemma~\ref{lem:relateduals}.
\end{itemize}
\begin{definition}\label{def:pivcat}
We say that a 2-category $\mathcal{C}$ with right duals is \emph{pivotal} if there is an invertible icon (Definition~\ref{def:pnt}) from the double duals pseudofunctor to the identity pseudofunctor. 
\end{definition}
\noindent
Roughly, the existence of an invertible icon in Definition~\ref{def:pivcat} comes down to the following statement:
\begin{itemize}
\item For every 1-morphism $X: r \to s$, there is a 2-isomorphism $\iota_X: X^{**} \to X$.
\item These $\{\iota_X\}$ are compatible with composition in $\mathcal{C}$.
\end{itemize}
In a pivotal 2-category, for any $X: r \to s$ the right dual $X^*$ is also a left dual for $X$ by the following cup and cap. Here and throughout we will indicate the double dual $X^{**}$ in the diagrammatic calculus by an $X$-labelled wire with a double upwards-pointing arrow:
\begin{calign}\label{eq:pivldualdef}
\includegraphics[valign=c]{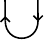}
~~:=~~
\includegraphics[valign=c]{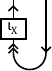}
&
\includegraphics[valign=c]{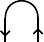}
~~:=~~
\includegraphics[valign=c]{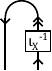}
\end{calign}
With these left duals, the left transpose of a 2-morphism is equal to the right transpose. Whenever we refer to a pivotal 2-category from now on, we suppose that the left duals are chosen in this way. 

There is a very useful graphical calculus for these compatible dualities in a pivotal 2-category.
To represent the transpose, we will modify our notation slightly. We now represent a morphism $f: X \to Y$ not by a rectangular box, but by a box where the right vertical edge is tilted:
\begin{align*}
\includegraphics[valign=c]{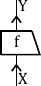}
\end{align*}
We now represent the transpose by rotating the boxes, as though we had `yanked' both ends of the wire in the RHS of~\eqref{eq:rtranspose}:
\begin{calign}\nonumber
\includegraphics[valign=c]{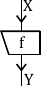}
:=
\includegraphics[valign=c]{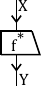}
\end{calign}
Using this notation, 2-morphisms now freely slide around cups and caps.
\begin{lemma}[{\cite[Lemma 3.12, Lemma 3.26]{Heunen2019}}]\label{lem:sliding}
Let $\mathcal{C}$ be a pivotal 2-category and $f:X \to Y$ a 2-morphism. Then:
\begin{calign}\label{eq:sliding}
\includegraphics[valign=c]{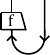}
~~=~~
\includegraphics[valign=c]{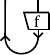}
&
\includegraphics[valign=c]{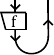}
~~=~~
\includegraphics[valign=c]{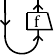}
&
\includegraphics[valign=c]{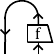}
~~=~~
\includegraphics[valign=c]{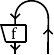}
&
\includegraphics[valign=c]{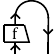}
~~=~~
\includegraphics[valign=c]{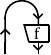}
\end{calign}
\end{lemma}
\noindent
In a pivotal 2-category, we can define notions of \emph{trace} and \emph{dimension} for 1-morphisms.  
\begin{definition}\label{def:trace}
Let $X: r \to s$ be an 1-morphism and let $f: X \to X$ be a 2-morphism in a pivotal 2-category $\mathcal{C}$. We define the \emph{right trace} of $f$ to be the following 2-morphism $\Tr_{R}(f): \id_r \to \id_r$: 
\begin{calign}\nonumber
\includegraphics{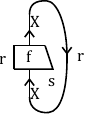}
\end{calign}
We define the \emph{right dimension} $\dim_R(X)$ of the 1-morphism $X$ to be $\Tr_R(\id_X)$. 
The \emph{left trace} $\Tr_L(f):\id_s \to \id_s$ and \emph{left dimension} $\dim_L(X)$ are defined analogously using the right cup and left cap. 
\end{definition}

\paragraph{Pivotal functors.} We now consider pseudofunctors between pivotal 2-categories. We first observe that the duals in $\mathcal{C}$ induce duals in $\mathcal{D}$ under a pseudofunctor $F: \mathcal{C} \to \mathcal{D}$. 
\begin{lemma}[Induced duals]\label{lem:indduals}
Let $X: r \to s$ be a 1-morphism in $\mathcal{C}$ and $[X^*,\eta,\epsilon]$ a right dual. Then $F(X^*)$ is a right dual of $F(X)$ in $\mathcal{D}$ with the following cup and cap:
\begin{calign}\nonumber
\includegraphics[scale=.6]{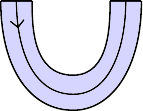}
&
\includegraphics[scale=.6]{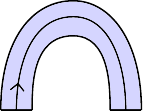}
\\\nonumber
F(\eta) & F(\epsilon)
\end{calign}
The analogous statement holds for left duals.
\end{lemma}
\begin{proof}
We show one of the snake equations~\eqref{eq:rightsnakes} in the case of right duals; the others are all proved similarly.
\begin{calign}\nonumber
\includegraphics[valign=c]{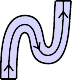}
~~=~~
\includegraphics[valign=c]{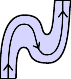}
~~=~~
\includegraphics[valign=c]{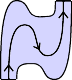}
~~=~~
\includegraphics[valign=c]{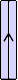}
\end{calign}
Here the first equality is by Lemma~\ref{lem:pushpast}, the second by~\eqref{eq:psfctnat} and the third by~\eqref{eq:psfctunital}.
\end{proof}
\noindent
For any 1-morphism $X$ of $\mathcal{C}$, then, we have two sets of left and right duals on $F(X)$; the first from the pivotal structure in $\mathcal{C}$ by Lemma~\ref{lem:indduals}, and the second from the pivotal structure in $\mathcal{D}$. To depict both dualities in the graphical calculus, we here introduce elements of the graphical syntax which allow us to `zoom in' and `zoom out', representing $F(X)$ as a directed coloured wire rather than as a boxed wire:
\begin{calign}\label{eq:zoominout}
\includegraphics[valign=c]{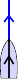}
&
\includegraphics[valign=c]{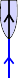}
\end{calign}
We emphasise that these elements of the graphical calculus are semantically empty, simply switching between two ways of representing $F(X)$. We can now represent the duality corresponding to the pivotal structure in $\mathcal{D}$ in the usual way on the directed coloured wire, writing $F(X)^* $ and $F(X)^{**}$ with a downwards and a double upwards arrow respectively, as usual.

We now define a pivotal pseudofunctor. 
Let $\mathcal{C},\mathcal{D}$ be pivotal 2-categories, and let $F: \mathcal{C} \to \mathcal{D}$ be a pseudofunctor. By Lemma~\ref{lem:relateduals}, for every 1-morphism $X: r \to s$ in $\mathcal{C}$ we obtain two 2-isomorphisms $F_l,F_r: F(X^*) \to F(X)^*$, the first from the left duality and the second from the right duality:
\begin{calign}\label{eq:pivlrisos}
\includegraphics[valign=c]{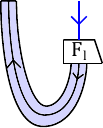}
~~=~~
\includegraphics[valign=c]{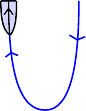}
&
\includegraphics[valign=c]{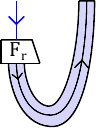}
~~=~~
\includegraphics[valign=c]{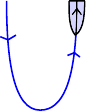}
\end{calign}
The following definition is inspired by the corresponding definition for monoidal functors~\cite[\S{}1.7.5]{Turaev2017}.
\begin{definition}\label{def:pivfct}
Let $\mathcal{C},\mathcal{D}$ be pivotal 2-categories, let $F: \mathcal{C} \to \mathcal{D}$ be a pseudofunctor, and let $F_l,F_r: F(X^*) \to F(X)^*$ be the isomorphisms~\eqref{eq:pivlrisos}. We say that $F$ is \emph{pivotal} if $F_l = F_r=:P$.
\end{definition}
\noindent
In the graphical calculus we again here write these isomorphisms $P$ and their inverses as `zoom ins' and `zoom outs', which this time are not semantically empty:
\begin{calign}\nonumber
\includegraphics[valign=c]{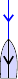}
~~:=~~
\includegraphics[valign=c]{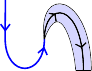}
&
\includegraphics[valign=c]{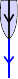}
~~:=~~
\includegraphics[valign=c]{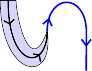}
\end{calign}
\ignore{
Uniqueness of right duals in $\mathcal{D}$ induces isomorphisms $F(X)^* \cong F(X^*)$ and $F(X)^{**} \cong F(X^{**})$, which we again write as zoom-ins and zoom-outs:
\begin{calign}
\includegraphics[valign=c]{Figures/svg/monoidalfunctors/fXdualzoomout.png}
~~:=~~
\includegraphics[valign=c]{Figures/svg/monoidalfunctors/fXdualzoomoutdef.png}
&
\includegraphics[valign=c]{Figures/svg/monoidalfunctors/fXdualzoomin.png}
~~:=~~
\includegraphics[valign=c]{Figures/svg/monoidalfunctors/fXdualzoomindef.png}
&
\includegraphics[valign=c]{Figures/svg/monoidalfunctors/fXdoubledualzoomout.png}
~~:=~~
\includegraphics[valign=c]{Figures/svg/monoidalfunctors/fXdoubledualzoomoutdef.png}
&
\includegraphics[valign=c]{Figures/svg/monoidalfunctors/fXdoubledualzoomin.png}
~~:=~~
\includegraphics[valign=c]{Figures/svg/monoidalfunctors/fXdoubledualzoomindef.png}
\end{calign}
It is easy to prove the following relations, which we will use later:
\begin{lemma}
Add
\end{lemma}
\begin{definition}
Let $\mathcal{C},\mathcal{D}$ be pivotal 2-categories. We say that a pseudofunctor $\mathcal{C} \to \mathcal{D}$ is \emph{pivotal} (a.k.a. \emph{planar pivotal}) when 
\begin{calign}
\includegraphics[valign=c]{Figures/svg/monoidalfunctors/pivfunctor1.png}
~~=~~
\includegraphics[valign=c]{Figures/svg/monoidalfunctors/pivfunctor2.png}
\end{calign}
\end{definition}
}

\subsection{Pivotal dagger 2-categories}
The final structure we will consider on a 2-category is a \emph{dagger}. In this section we define a dagger 2-category and discuss compatibility with the various notions already introduced. 
\begin{definition}\label{def:dagcat}
A dagger 2-category is a 2-category equipped with contravariant identity-on-objects functors $\dagger_{r,s}:\mathcal{C}(r,s) \to \mathcal{C}(r,s)$ for each pair of objects $r,s$, which are:
\begin{itemize} 
\item \emph{Involutive}: for any morphism $f: X \to Y$ in $\mathcal{C}(r,s)$, $\dagger_{r,s}(\dagger_{r,s}(f)) = f$. (This is to say that $\mathcal{C}(r,s)$ is a \emph{dagger category}.)
\item \emph{Compatible with 1-morphism composition}: for any 1-morphisms $X,X': r \to s$ and $Y,Y': s \to t$, and 2-morphisms $\alpha: X \to X'$ and $\beta: Y \to Y'$, we have $(\alpha \otimes \beta)^{\dagger_{r,t}} = \alpha^{\dagger_{r,s}}  \otimes \beta^{\dagger_{s,t}}$.
\end{itemize}
\ignore{
\item \emph{Compatible with composition}: For any 2-morphisms $f:X \to Y,g:Y \to Z$ in $\mathcal{C}(r,s)$ and $g$ in $\mathcal{C}(s,t)$, $(f \circ g)^{\dagger_{r,t}} = f^{\dagger_{r,s}} \circ g^{\dagger_{s,t}}$.
\item \emph{Preserves the identity 2-morphisms.} and $(\id_{X})^{\dagger_{r,s}} = \id_X$.
\end{itemize} 
}
We call the image of a 2-morphism $f:X \to Y$ under $\dagger_{r,s}$ its \emph{dagger}, and write it as $f^{\dagger}$.
\end{definition}
\noindent
In the graphical calculus, we represent the dagger of a 2-morphism by reflection in a horizontal axis, preserving the direction of any arrows:
\begin{calign}\label{eq:graphcalcdagger}
\includegraphics[valign=c]{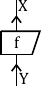}
~~:=~~
\includegraphics[valign=c]{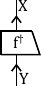}
\end{calign}
\begin{definition}
Let $\mathcal{C}$ be a dagger 2-category.  We say that a 2-morphism $\alpha: X \to Y$ in $\mathcal{C}(r,s)$ is an \emph{isometry} if $\alpha^{\dagger} \circ \alpha  = \id_X$. We say that it is \emph{unitary} if it is an isometry and additionally a \emph{coisometry}, i.e. $\alpha \circ \alpha^{\dagger} = \id_Y$.
\end{definition}
\begin{definition}\label{def:daggerequiv}
Let $\mathcal{C}$ be a dagger 2-category and let $r,s$ be objects. We say that a 1-morphism $X: r \to s$ is a \emph{dagger equivalence} if there exists some 1-morphism $X^{-1}: s \to r$ (called the weak inverse) and unitary 2-morphisms $\eta: \id_s \to X^{-1} \otimes X$ and $\epsilon: X \otimes X^{-1} \to \id_r$ witnessing an equivalence (Definition~\ref{def:equiv}). It is a standard result that $\eta,\epsilon$ may be chosen such that $[X^{-1},\eta,\epsilon]$ is a right dual for $X$ (this is to say that any dagger equivalence can be promoted to an \emph{adjoint} dagger equivalence).
\end{definition}
\noindent
We now give the condition for  compatibility of dagger and pivotal structure.
\begin{definition}\label{def:pivdagcat}
Let $\mathcal{C}$ be a pivotal 2-category which is also a dagger 2-category. \ignore{Let $X: r \to s$ be a 1-morphism in $\mathcal{C}$, and and let $[X^*,\eta,\epsilon]$ be the chosen right dual; we represent $\eta,\epsilon$ in the graphical calculus by a cup and a cap.} We say that $\mathcal{C}$ is a  \emph{pivotal dagger 2-category} when, for all 1-morphisms $X: r \to s$:
\begin{calign}\label{eq:pivdagcat}
\includegraphics[valign=c]{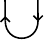}
~~=~~
\left(
\includegraphics[valign=c]{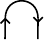}\right)^{\dagger}
&
\includegraphics[valign=c]{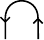}
~~=~~
\left(\includegraphics[valign=c]{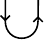}\right)^{\dagger}
\end{calign}
\end{definition}
\begin{remark}
For any object $X$ in a dagger 2-category, a right dual $[X^*,\eta_X,\epsilon_X]$ is also a left dual $[X^*,\epsilon_X^{\dagger},\eta_X^{\dagger}]$. This means that a dagger 2-category with right duals also has left duals. The pivotal structure gives another way to obtain left duals from right duals~\eqref{eq:pivldualdef}.  The equation~\eqref{eq:pivdagcat} implies that the left duals obtained from the dagger structure are the same as those obtained from the pivotal structure.

Practically, when taking the dagger of a cup or a cap in a pivotal dagger category, the equation~\eqref{eq:pivdagcat} implies we should reflect the cup or cap in a horizontal axis, preserving the direction of the arrows.
\end{remark}
The following result from the theory of pivotal dagger categories generalises immediately to pivotal dagger 2-categories, since the proof is entirely diagrammatic. 
\begin{lemma}[{\cite[Prop. 3.5.2, Prop. 3.5.3]{Heunen2019}}]\label{lem:pivdagisometry}
Let $\mathcal{C}$ be a pivotal dagger 2-category. Then the 2-isomorphism components $\iota_X$ of the invertible icon $\iota: **_{\mathcal{C}} \to \id_{\mathcal{C}}$ are unitary, and the following equality holds:
\begin{calign}
\includegraphics[valign=c]{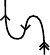}
~~=~~
\includegraphics[valign=c]{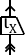}
&&
\includegraphics[valign=c]{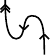}
~~=~~
\includegraphics[valign=c]{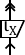}
\end{calign}
\end{lemma}
\noindent
For any morphism $f:X \to Y$, a pivotal dagger structure implies the following \emph{conjugate} 2-morphism $f_*$ is graphically well-defined:
\begin{calign}\label{eq:conjugate}
\includegraphics[valign=c]{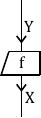}
~~:=~~
\includegraphics[valign=c]{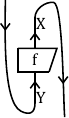}
~~=~~
\includegraphics[valign=c]{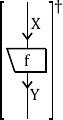}
\end{calign}
\begin{remark}\label{rem:sliding2}
In view of~\eqref{eq:pivdagcat}, in a pivotal dagger 2-category we also have sliding equations for the dagger and conjugate 2-morphisms obtained by taking the reflections of~\eqref{eq:sliding} in a horizontal axis.
\end{remark}
\noindent
Finally, we consider the notion of a unitary pseudofunctor between dagger 2-categories.
\begin{definition}
Let $\mathcal{C},\mathcal{D}$ be dagger 2-categories and let $F: \mathcal{C} \to \mathcal{D}$ be a pseudofunctor. We say that $F$ is \emph{unitary} if the following hold:
\begin{itemize}
\item For any 2-morphism $f$, $F(f^{\dagger}) = F(f)^{\dagger}$:
\begin{calign}\nonumber
\includegraphics[scale=0.5,valign=c]{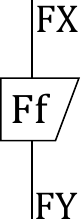}
~~=~~
\includegraphics[scale=0.5,valign=c]{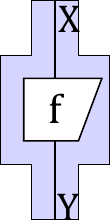}
\end{calign}
\item The multiplicators $\{m_{X,Y}\}$ and unitors $\{u_r\}$ are all unitary 2-morphisms in $\mathcal{D}$.  
\end{itemize}
\end{definition}
\begin{remark}
The latter condition implies that our depiction of the inverses $\{m^{-1}_{X,Y}\}$ and $\{u_r^{-1}\}$ by reflection in a horizontal axis~(\ref{eq:multiplicator},~\ref{eq:unitor}) is consistent with the `horizontal flip' calculus~\eqref{eq:graphcalcdagger} of the dagger in $\mathcal{D}$.
\end{remark}

\section{Dualisable pseudonatural transformations}\label{sec:duals}

\subsection{Duals}
\ignore{
It is worth remarking that most of the results in the following sections can be immediately horizontally categorified to apply to pseudonatural transformations between arbitrary 2-categories. However, we here restrict statements to monoidal category theory.} 
We now consider invertibility of pseudonatural transformations. As we saw in Definition~\ref{def:duals}, the most general notion of invertibility of a 1-morphism in a 2-category is dualisability. The following definition is nothing more than an explicit statement of what it means for a 1-morphism in $\Fun(\mathcal{C},\mathcal{D})$ to have a dual (Definition~\ref{def:duals}).
\begin{definition}\label{def:dualpnt}
Let $F, G: \mathcal{C} \to \mathcal{D}$ be pseudofunctors and $\alpha: F \to G$ a pseudonatural transformation. A \emph{right dual} for $\alpha$ is a triple $[\alpha^*, \eta, \epsilon]$, where  $\alpha^*: G \to F$ is a pseudonatural transformation and $\epsilon: \alpha \otimes \alpha^* \to \id_F$ and $\eta: \id_G \to \alpha^* \otimes \alpha$ are modifications, such that the following equations hold for any 1-morphism $X: r \to s$ in $\mathcal{C}$:
\begin{calign}\label{eq:rightmodsnakes}
\includegraphics[valign=c]{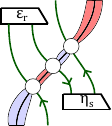}
~~=~~
\includegraphics[valign=c]{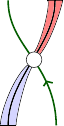}
&
\includegraphics[valign=c]{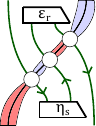}
~~=~~
\includegraphics[valign=c]{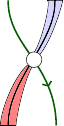}
\end{calign}
In the above equations we have drawn the $\alpha$-wire in green with an upwards-facing arrow and the $\alpha^*$-wire in green with a downwards-facing arrow, as though $\alpha_r$ and $\alpha_r^*$ were dual 1-morphisms. This will be justified by Lemma~\ref{lem:pntduals}. 
A \emph{left dual} is defined analogously. 

\ignore{
if there exist modifications $\nu: (\beta \circ \alpha, A \otimes B) \to (\id,\mathbbm{1})$ and $\epsilon:  (\id,\mathbbm{1}) \to (\alpha \circ \beta, B \otimes A)$ such that the following equations hold:}
\ignore{
We say that $\nu, \epsilon$ define an \emph{equivalence} if the modifications $d,e$ are invertible.\footnote{For simplicity, since every equivalence in a 2-category can be promoted to an adjoint equivalence~\cite{}, we have here considered equivalence as a form of adjunction.} In this case, }
\end{definition}
\begin{lemma}\label{lem:pntduals}
Let $F,G: \mathcal{C} \to \mathcal{D}$ be pseudofunctors and $\alpha: F \to G$ a pseudonatural transformation with right dual $[\alpha^*,\eta,\epsilon]$. Then for each object $r$ of $\mathcal{C}$, $[\alpha^*_r,\eta_r,\epsilon_r]$ is a right dual for $\alpha_r$ in $\mathcal{D}$. The analogous statement holds for left duals.
\end{lemma}
\begin{proof}
We prove the right snake equation for right duals; everything else may be proved similarly. For any object $r$ of $\mathcal{C}$:
\begin{calign}
\includegraphics[valign=c]{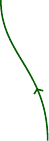}
~~=~~
\includegraphics[valign=c]{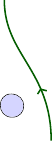}
~~=~~
\includegraphics[valign=c]{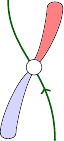}
~~=~~
\includegraphics[valign=c]{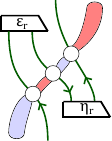}
~~=~~
\includegraphics[valign=c]{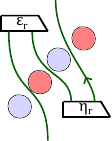}
~~=~~
\includegraphics[valign=c]{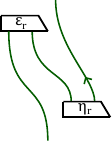}
\end{calign}
Here the first equation is by invertibility of the unitor $u_r$~\eqref{eq:unitor} for $F$; the second by monoidality~\eqref{eq:pntmonunit} of the pseudonatural transformation $\alpha$ on the 1-morphism $\id_r: r \to r$ and invertibility of the unitor for $G$; the third by~\eqref{eq:rightmodsnakes}; the fourth by monoidality of $\alpha$ and $\alpha^*$ on $\id_r$; and the last by invertibility of the unitors.
\end{proof}
\noindent
From this point forward, therefore, we will draw $\eta_r$ and $\epsilon_r$ as a cup and cap.  If $\mathcal{C}$ has duals, we obtain explicit expressions for the left and right duals in $\Fun(\mathcal{C},\mathcal{D})$ whenever they exist.
\begin{lemma}\label{lem:pntduals}
Let $F, G: \mathcal{C} \to \mathcal{D}$ be pseudofunctors, and suppose that $\mathcal{C}$ has left duals. A pseudonatural transformation $\alpha: F \to G$ has a right dual in $\Fun(\mathcal{C},\mathcal{D})$ precisely when $\alpha_r$ has a right dual $[\alpha_r^*,\eta_r, \epsilon_r]$ in $\mathcal{D}$ for each object $r$ of $\mathcal{C}$. A right dual $\alpha^*$ is defined as follows:
\begin{itemize}
\item For each object $r$ of $\mathcal{C}$, $(\alpha^*)_r = (\alpha_r)^*$ and the components of the modifications $\eta, \epsilon$ are $[\eta_r,\epsilon_r]$.
\item For each 1-morphism $X: r \to s$ of $\mathcal{C}$, $(\alpha^*)_X$ is:
\begin{calign}\label{eq:dualpnt}
\includegraphics[valign=c]{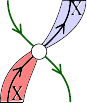}
:=
\includegraphics[valign=c]{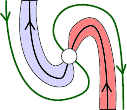}
\end{calign}
\end{itemize}
This statement also holds with `left' and `right' swapped, in which case a left dual ${}^*\alpha$ is defined as follows:
\begin{itemize}
\item For each object $r$ of $\mathcal{C}$, $({}^*\alpha)_r = {}^*(\alpha_r)$ and the components of the modifications $\eta, \epsilon$ are $[\eta_r,\epsilon_r]$.
\item For each 1-morphism $X: r \to s$ of $\mathcal{C}$, $({}^*\alpha)_X$ is defined as in~\eqref{eq:dualpnt}, but with the opposite transposition.
\end{itemize}
\end{lemma}
\begin{proof}
We consider the case of the right dual $\alpha^*$; the argument for the left dual is similar. 

If some $\alpha_r$ has no right dual, then nor can $\alpha$ by Lemma~\ref{lem:pntduals}. 

If every $\alpha_r$ has some right dual, then we must show firstly that $\alpha^*$ as defined is a pseudonatural transformation, and secondly that $\eta, \epsilon$ as defined are modifications satisfying the snake equations~\eqref{eq:rightsnakes}.
\begin{enumerate}
\item \emph{Naturality of $\alpha^*$.~\eqref{eq:pntnat}} For all 2-morphisms $f: X \to Y$ in $\mathcal{C}$:
\begin{calign}
\includegraphics[valign=c]{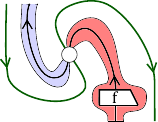}
~~=~~
\includegraphics[valign=c]{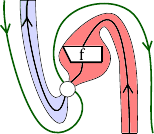}
~~=~~
\includegraphics[valign=c]{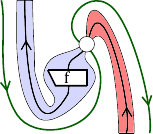}
~~=~~
\includegraphics[valign=c]{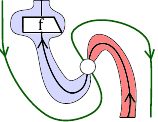}
\end{calign}
Here the first and third equalities use Lemma~\ref{lem:sliding}; the second equality is by naturality of $\alpha$.
\item \emph{Monoidality of $\alpha^*$.} (\ref{eq:pntmon}-\ref{eq:pntmonunit}) 
\begin{itemize}
\item For every pair of 1-morphisms $X:r \to s, Y: s \to t$ in $\mathcal{C}$, let $f: {}^*(X \otimes Y) \to {}^*Y \otimes {}^*X$ be the isomorphism of Lemmas~\ref{lem:nestedduals} and~\ref{lem:relateduals}. Then:
\begin{calign}\nonumber
\includegraphics[valign=c]{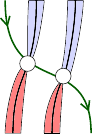}
~~=~~
\includegraphics[valign=c]{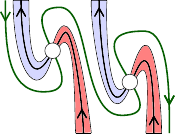}
~~=~~
\includegraphics[valign=c]{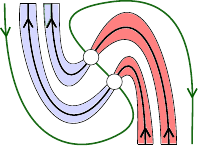}
~~=~~
\includegraphics[valign=c]{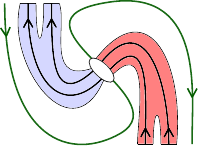}\\\label{eq:pntdualmonpf}
~~=~~
\includegraphics[valign=c]{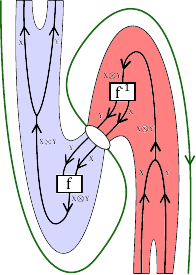}
~~=~~
\includegraphics[valign=c]{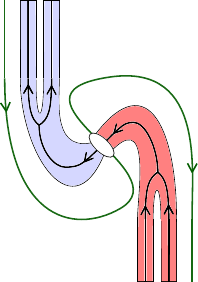}
~~=~~
\includegraphics[valign=c]{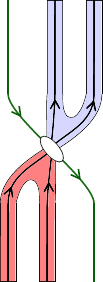}
\end{calign} 
Here the first equality is by definition; the second by a snake equation for $\alpha_s$; the third by monoidality of $\alpha$ and some manipulation of functorial boxes; the fourth by Lemmas~\ref{lem:nestedduals} and~\ref{lem:relateduals}; the fifth by naturality of $\alpha$; and the sixth by definition.
\item For every object $r$ of $\mathcal{C}$:
\begin{calign}
\includegraphics[valign=c]{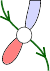}
~~=~~
\includegraphics[valign=c]{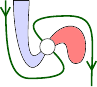}
~~=~~
\includegraphics[valign=c]{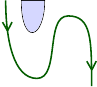}
~~=~~
\includegraphics[valign=c]{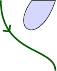}
\end{calign} 
Here the first equality is by definition, the second by monoidality of $\alpha$ and manipulation of functorial boxes, and the third by a snake equation for $\alpha_r$. We have assumed for that the chosen left dual of $\id_r$ is $[\id_r,\id_{\id_r},\id_{\id_r}]$; in general one can use Lemma~\ref{lem:relateduals} and naturality of $\alpha$ as in~\eqref{eq:pntdualmonpf}.
\end{itemize}
\item Since $\eta_r,\epsilon_r$ already satisfy the snake equations for every $r$ by assumption, we need only show that $\eta, \epsilon$ are modifications. For all $X: r \to s$ in $\mathcal{C}$:
\begin{calign}
\includegraphics[valign=c]{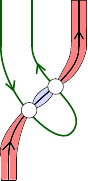}
~~=~~
\includegraphics[valign=c]{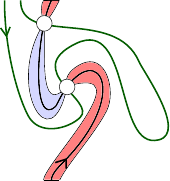}
~~=~~
\includegraphics[valign=c]{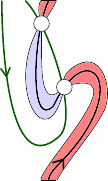}
~~=~~
\includegraphics[valign=c]{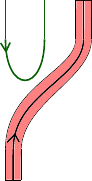}
\end{calign}
\begin{calign}
\includegraphics[valign=c]{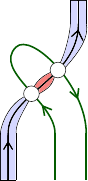}
~~=~~
\includegraphics[valign=c]{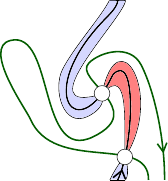}
~~=~~
\includegraphics[valign=c]{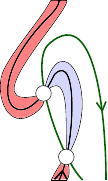}
~~=~~
\includegraphics[valign=c]{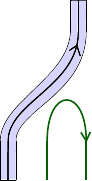}
\end{calign}
Here, the first equalities are by definition, the second are by a snake equation for $\alpha_r^*$ or $\alpha_s^*$, and the third are by naturality and monoidality of $\alpha$. 
\end{enumerate}
\end{proof}
\begin{corollary}\label{cor:dualsexistence}
If $\mathcal{C}$ has left duals, and $\mathcal{D}$ has right duals, then $\Fun(\mathcal{C},\mathcal{D})$ has right duals. 
This statement also holds with `left' and `right' swapped.
\end{corollary}

\subsection{Pivotality}
We now consider pivotality of $\Fun(\mathcal{C},\mathcal{D})$. Recall that a 2-category with right duals is \emph{pivotal} (Definition~\ref{def:pivcat}) if there is an invertible icon (Definition~\ref{def:pnt}) from the double duals pseudofunctor to the identity pseudofunctor. 

We now show that $\Fun(\mathcal{C},\mathcal{D})$ inherits pivotality from $\mathcal{C}$ and $\mathcal{D}$ upon restriction to pivotal pseudofunctors. 
\begin{definition}
When $\mathcal{C},\mathcal{D}$ are pivotal we define $\Fun_p(\mathcal{C},\mathcal{D}) \subset \Fun(\mathcal{C},\mathcal{D})$ to be the subcategory whose objects are pivotal pseudofunctors.
\end{definition}
\begin{proposition}\label{prop:pivinduced}
Let $\mathcal{C}, \mathcal{D}$ be pivotal 2-categories, and let $\iota: {**}_{\mathcal{D}} \to \id_{\mathcal{D}}$ be the pivotal structure on $\mathcal{D}$.

The 2-category $\Fun_p(\mathcal{C},\mathcal{D})$ has a canonical structure of a pivotal 2-category. The pivotal structure $\hat{\iota}: {**}_{\Fun(\mathcal{C},\mathcal{D})} \to \id_{\Fun(\mathcal{C},\mathcal{D})}$ assigns to every pseudonatural transformation $\alpha^{**}: F \to G$ the invertible modification $\hat{\iota}_{\alpha}: \alpha^{**} \to \alpha$ 
\ignore{:\alpha^{**} \to (\alpha^{**})^{\hat{\iota}_\alpha} with source $\alpha^{**}$ }
whose components are the 2-isomorphisms $\iota_{\alpha_r}: \alpha_r^{**} \to \alpha_r$ from the pivotal structure on $\mathcal{D}$.
\end{proposition}
\begin{proof}
First we show that the $\hat{\iota}_{\alpha}$ are really modifications. Since $\{\iota_{\alpha_r}\}$ are 2-isomorphisms it is immediate that the \emph{$\hat{\iota}_{\alpha}$-conjugate} $(\alpha^{**})^{\hat{\iota}_{\alpha}}$ of $\alpha^{**}$ is a pseudonatural transformation $F \to G$, where $(\alpha^{**})^{\hat{\iota}_{\alpha}}_r = \alpha_r$ for all objects $r$ of $\mathcal{C}$, and $(\alpha^{**})^{\hat{\iota}_{\alpha}}_X$ is defined as follows for all $X: r \to s$:
\begin{calign}\label{eq:conjpnt}
\includegraphics[scale=0.7,valign=c]{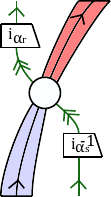}
\end{calign}
It is also clear that $\hat{\iota}_{\alpha}$ is a modification $\alpha^{**} \to (\alpha^{**})^{\hat{\iota}_{\alpha}}$.

We now show that $\hat{\iota}_{\alpha}$ has the right target, i.e. $(\alpha^{**})^{\hat{\iota}_{\alpha}} = \alpha$. 
We first observe that the chosen left dual of a pseudonatural transformation between pivotal functors is identical to its chosen right dual: 
\begin{calign}\label{eq:ldualisrdual}
\includegraphics[scale=0.5,valign=c]{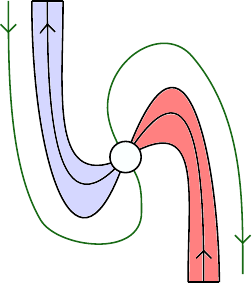}
~~=~~
\includegraphics[scale=0.5,valign=c]{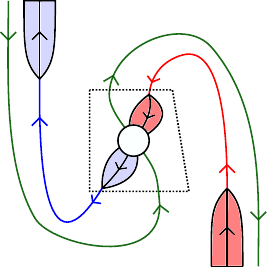}
~~=~~
\includegraphics[scale=0.5,valign=c]{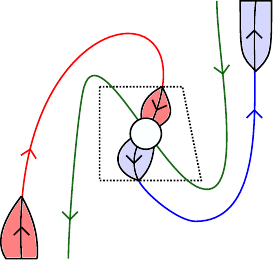}
~~=~~
\includegraphics[scale=0.5,valign=c]{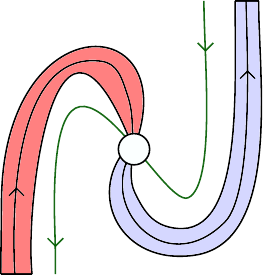}
\end{calign}
Here for the first and third equalities we used Lemma~\ref{lem:relateduals} and the `zoom out' notation~\eqref{eq:zoominout} to relate the duals in $\mathcal{C}$ and $\mathcal{D}$. For the second equality we follow the custom in the setting of pivotal categories of appealing to an unproven but very plausible coherence theorem~\cite[Theorem 4.14]{Selinger2010}; it is not hard to prove the equality directly from the axioms, but we leave this to the reader. For the third equality we require that the pseudofunctors are pivotal.

Now for any $\alpha: F \to G$ and $X: r \to s$ in $\mathcal{C}$ we have:
\begin{calign}\nonumber
\includegraphics[scale=1,valign=c]{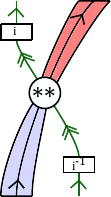}
~~=~~
\includegraphics[scale=0.5,valign=c]{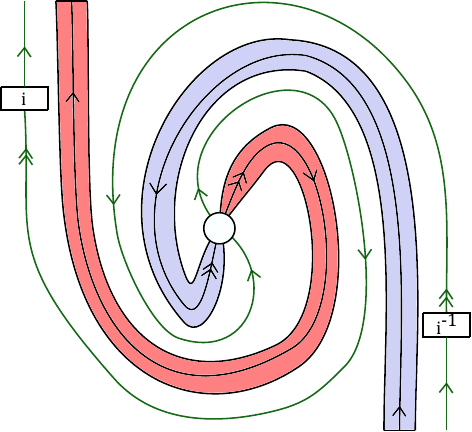}
~~=~~
\includegraphics[scale=0.5,valign=c]{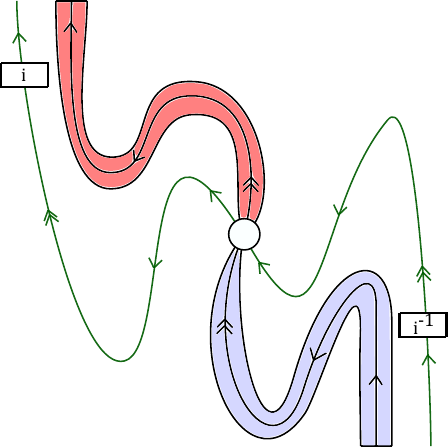}
\\\nonumber
~~=
\includegraphics[scale=0.5,valign=c]{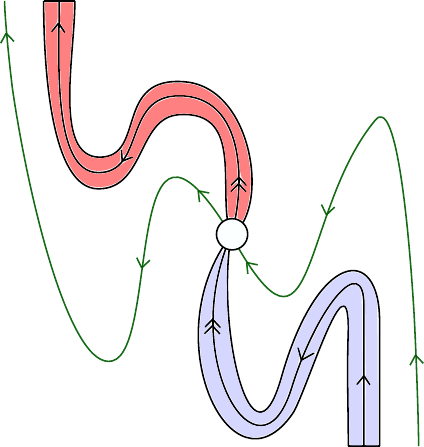}
~~=~~
\includegraphics[scale=0.5,valign=c]{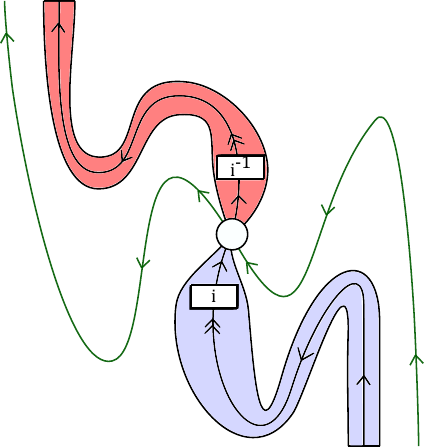}
~~=~~
\includegraphics[scale=0.5,valign=c]{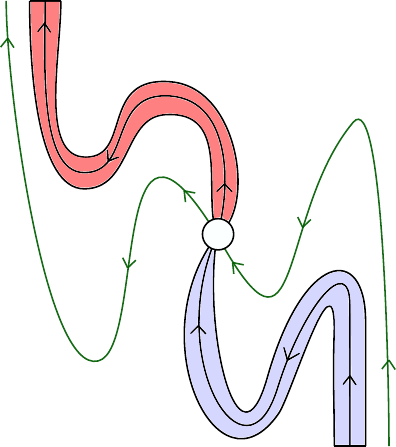}
\\\nonumber
~~=~~
\includegraphics[scale=0.7,valign=c]{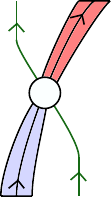}
\end{calign}
Here the first equality is by definition; the second uses~\eqref{eq:ldualisrdual}; the third uses the definition~\eqref{eq:pivldualdef} of the left duality in the pivotal 2-category $\mathcal{D}$; the fourth uses naturality of $\alpha$ to insert $\iota  \iota^{-1}$, where $\iota: X^{**} \to X$ is the isomorphism from the pivotal structure in $\mathcal{C}$; the fifth uses the definition~\eqref{eq:pivldualdef} of the left duality in $\mathcal{C}$; and the last uses the snake equations in $\mathcal{C}$ and $\mathcal{D}$. 

Finally, we need to show that $\hat{\iota}$ is an invertible icon ${**}_{\Fun(\mathcal{C},\mathcal{D})} \to \id_{\Fun(\mathcal{C},\mathcal{D})}$.
\begin{itemize}
\item \emph{Monoidality}: For every pair of pseudonatural transformations $\alpha: F \to G$, $\beta: G \to H$, we need $\hat{\iota}_{\alpha \otimes \beta} = \hat{\iota}_{\alpha} \otimes \hat{\iota}_{\beta}$. For each $X: r \to s$ this is implied by monoidality of $\iota: **_{\mathcal{D}} \to \id_{\mathcal{D}}$. Indeed, we have:
$$
(\hat{\iota}_{\alpha \otimes \beta})_r
= \iota_{\alpha_r \otimes \beta_r}
= \iota_{\alpha_r} \otimes \iota_{\beta_r} = (\hat{\iota}_{\alpha})_r \otimes (\hat{\iota}_{\beta})_r
$$
\item \emph{Naturality}: We need that, for every modification $f: \alpha \to \beta$, $\hat{\iota}_{\beta} \circ f^{**} = f \circ \hat{\iota}_{\alpha}$. For each $X: r \to s$ this is implied by naturality of $\iota: **_{\mathcal{D}} \to \id_{\mathcal{D}}$. Indeed, we have:
$$
(f \circ \hat{\iota}_{\alpha})_r = (f)_r \circ (\hat{\iota}_\alpha)_r
= f_r \circ \iota_{\alpha_r} = \iota_{\beta_r} \circ f_r = (\hat{\iota}_{\beta})_r \circ f_r
= (\hat{\iota}_{\beta} \circ f)_r
$$
\end{itemize}
\end{proof}

\section{Unitary pseudonatural transformations}\label{sec:unitary}

We have considered the case where $\mathcal{C},\mathcal{D}$ are pivotal. We now consider the case where $\mathcal{C},\mathcal{D}$ are pivotal dagger and the pseudofunctors are unitary.

In this case, we get a new contravariant operation on pseudonatural transformations.

\begin{lemma}
Let $F,G: \mathcal{C} \to \mathcal{D}$ be unitary pseudofunctors between pivotal dagger 2-categories. For any pseudonatural transformation $\alpha: F \to G$, there is a pseudonatural transformation $\alpha^{\dagger}: G \to F$ (its \emph{dagger}), defined as follows:
\begin{itemize}
\item For each object $r$ of $\mathcal{C}$, $(\alpha^{\dagger})_r = (\alpha_r)^*$.
\item For each $X: r \to s$ in $\mathcal{C}$, $(\alpha^{\dagger})_X$ is defined as follows:
\begin{calign}\label{eq:daggerpnt}
\includegraphics[scale=0.7,valign=c]{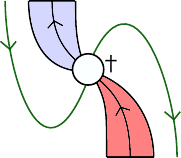}
\end{calign}
\end{itemize}
is also a pseudonatural transformation.
\end{lemma}
\begin{proof}
We must show naturality and monoidality.
\begin{itemize}
\item \emph{Naturality.} For any $f: X\to Y$ in $\mathcal{C}$:
\begin{calign}\nonumber
\includegraphics[scale=0.7,valign=c]{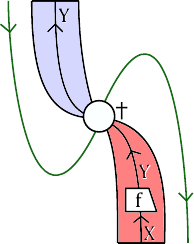}
~~=~~
\includegraphics[scale=0.7,valign=c]{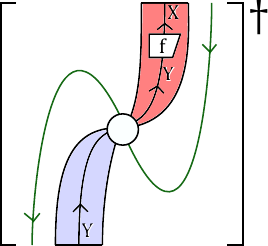}
~~=~~
\includegraphics[scale=0.7,valign=c]{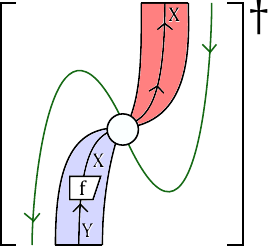}
~~=~~
\includegraphics[scale=0.7,valign=c]{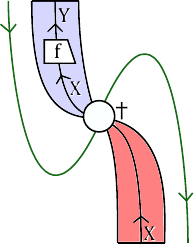}
\end{calign}
Here the first equality is by unitarity of $G$, the second equality is by naturality of $\alpha$, and the third equality is by unitarity of $F$.
\item \emph{Monoidality.} For any $X: r \to s, Y: s \to t$ in $\mathcal{C}$:
\begin{calign}\nonumber
\includegraphics[scale=0.7,valign=c]{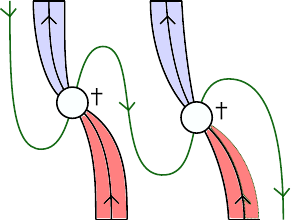}
~~=~~
\includegraphics[scale=0.7,valign=c]{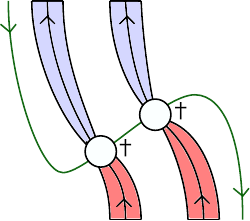}
~~=~~
\includegraphics[scale=0.7,valign=c]{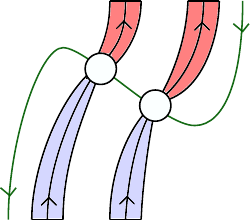}\\\nonumber
~~=~~
\includegraphics[scale=0.7,valign=c]{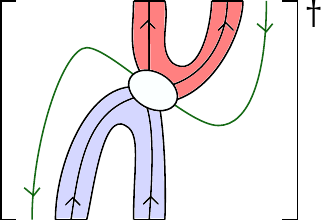}
~~=~~
\includegraphics[scale=0.7,valign=c]{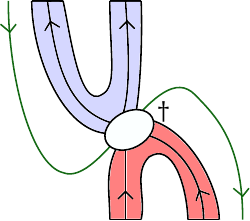}
\end{calign}
Here the first and second equalities are by dagger pivotality of $\mathcal{D}$, the third equality is by monoidality of $\alpha$, and the fourth equality is by unitarity of $F,G$ and dagger pivotality of $\mathcal{D}$. 

We leave the other monoidality condition~\eqref{eq:pntmonunit} to the reader. 
\end{itemize}
\end{proof}
\noindent
We would like $\Fun(\mathcal{C},\mathcal{D})$ to inherit the structure of a dagger  2-category. In general, however, there is no reason why the componentwise dagger of a modification $f: \alpha \to \beta$ --- the only reasonable candidate for a dagger on $\Fun(\mathcal{C},\mathcal{D})$ --- should yield a modification $f^{\dagger}: \beta \to \alpha$.

This problem is resolved by restriction to unitary pseudonatural transformations. There are  two obvious ways to define unitarity. First, given that the dual is the `inverse' of a pseudonatural transformation, we could ask that the dagger~\eqref{eq:daggerpnt} of the transformation should be equal to the right dual~\eqref{eq:dualpnt}. Alternatively, by analogy with the definition of unitary monoidal natural transformations, and motivated by  physicality in quantum mechanics~\cite{Vicary2012}, we might demand that the components of the transformation be individually unitary in $\mathcal{D}$. In fact, these definitions are equivalent.
\begin{lemma}\label{lem:unitaritydefsequiv}
Let $\mathcal{C},\mathcal{D}$ be pivotal dagger 2-categories and let $\alpha: F \to G$ be a pseudonatural transformation between unitary pseudofunctors $F,G: \mathcal{C} \to \mathcal{D}$. The following are equivalent:
\begin{enumerate}
\item There is an equality of pseudonatural transformations $\alpha^* = \alpha^{\dagger}$.
\item For all 1-morphisms $X: r \to s$ in $\mathcal{C}$, the component $\alpha_X: F(X) \otimes \alpha_s \to \alpha_r \otimes G(X)$ is unitary.
\end{enumerate}
\end{lemma} 
\begin{proof}
(i) $\Rightarrow$ (ii): For all $X: r \to s$ in $\mathcal{C}$, unitarity of $\alpha_X$ follows from right duality:
\begin{calign}\nonumber
\includegraphics[valign=c]{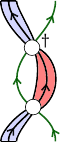}
~~=~~
\includegraphics[valign=c]{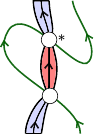}
~~=~~
\includegraphics[valign=c]{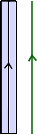}
&
\includegraphics[valign=c]{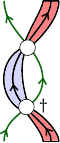}
~~=~~
\includegraphics[valign=c]{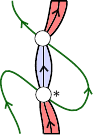}
~~=~~
\includegraphics[valign=c]{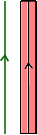}
\end{calign}
(ii) $\Rightarrow$ (i): Unitarity of the components implies that $[\alpha^{\dagger},\eta,\epsilon]$ is a right dual, where $\eta,\epsilon$ are the cup and cap of the right dual $[\alpha^*,\eta,\epsilon]$, since for each component:
\begin{calign}\nonumber
\includegraphics[valign=c]{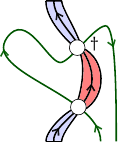}
~~=~~
\includegraphics[valign=c]{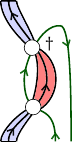}
~~=~~
\includegraphics[valign=c]{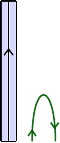}
&&
\includegraphics[valign=c]{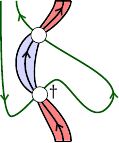}
~~=~~
\includegraphics[valign=c]{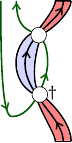}
~~=~~
\includegraphics[valign=c]{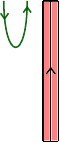}
\end{calign}
\noindent
But this implies equality $\alpha^{\dagger} = \alpha^{*}$; indeed, since the cup and cap modifications are identical, the unique 2-isomorphism of Lemma~\ref{lem:relateduals} relating the two right duals in $\Fun(\mathcal{C},\mathcal{D})$ must be the identity.
\end{proof}
\noindent
We therefore make the following definition.
\begin{definition}
Let $\mathcal{C},\mathcal{D}$ be pivotal dagger 2-categories and let $F,G: \mathcal{C} \to \mathcal{D}$ be unitary pseudofunctors. Then a \emph{unitary pseudonatural transformation (UPT)} $\alpha: F \to G$ is a pseudonatural transformation such that either of the following equivalent conditions are satisfied:
\begin{itemize}
\item There is an equality of pseudonatural transformations $\alpha^* = \alpha^{\dagger}$.
\item For all 1-morphisms $X: r \to s$ in $\mathcal{C}$, the component $\alpha_X: F(X) \otimes \alpha_s \to \alpha_r \otimes G(X)$ is unitary.
\end{itemize} 
\end{definition}
\noindent
\begin{definition}
When $\mathcal{C},\mathcal{D}$ are pivotal dagger, let $\Fun_u(\mathcal{C},\mathcal{D}) \subset \Fun(\mathcal{C},\mathcal{D})$ be the subcategory  whose objects are unitary pseudofunctors and whose 1-morphisms are UPTs.
\end{definition}
\noindent
We now show that $\Fun_u(\mathcal{C},\mathcal{D})$ is a dagger 2-category.  Moreover, it is pivotal dagger, with no need to restrict to pivotal functors. 
\begin{theorem}\label{thm:funcddagger}
Let $\mathcal{C},\mathcal{D}$ be pivotal dagger 2-categories. Then the 2-category $\Fun_u(\mathcal{C},\mathcal{D})$ is pivotal dagger, where: 
\begin{itemize}
\item The dagger of a modification $f: \alpha \to \beta$ is defined on components as $(f^{\dagger})_r = (f_r)^{\dagger}$. 
\item The pivotal structure $\hat{\iota}: **_{\Fun_u(\mathcal{C},\mathcal{D})} \to \id_{\Fun_u(\mathcal{C},\mathcal{D})}$ assigns to every pseudonatural transformation $\alpha^{**}: F \to G$ the invertible modification $\hat{\iota}_{\alpha}: \alpha^{**} \to \alpha$ whose components are the 2-isomorphisms $\iota_{\alpha_r}: \alpha_r^{**} \to \alpha_r$ from the pivotal structure on $\mathcal{D}$.
\end{itemize}
\end{theorem}
\begin{proof}
We first show that $f^{\dagger}$ is a modification $\beta \to \alpha$: 
\begin{calign}\nonumber
\includegraphics[valign=c]{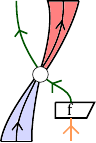}
~~=~~
\includegraphics[valign=c]{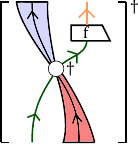}
~~=~~
\includegraphics[valign=c]{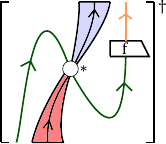}
~~=~~
\includegraphics[valign=c]{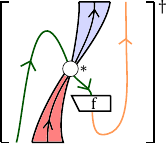}\\\nonumber
~~=~~
\includegraphics[valign=c]{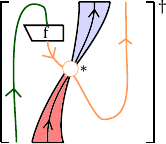}
~~=~~
\includegraphics[valign=c]{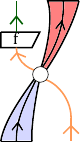}
\end{calign}
Here the second equality is by unitarity of $\alpha$, and the fourth equality is by transposition in $\Fun_u(\mathcal{C},\mathcal{D})$.

$\Fun_u(\mathcal{C},\mathcal{D})$ is therefore a dagger 2-category. We now show that it is pivotal dagger. First we demonstrate that $\hat{\iota}$ is indeed a pivotal structure. Since by Lemma~\ref{lem:unitaritydefsequiv} we have $\alpha^* = \alpha^{\dagger}$, we have the following expression for the components of $\alpha^{**}= \alpha^{\dagger \dagger}$:
\begin{calign}\label{eq:doubledagger}
\includegraphics[valign=c]{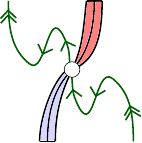}
~~=~~
\includegraphics[valign=c]{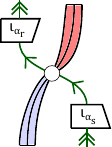}
\end{calign}
Here the last equality is by Lemma~\ref{lem:pivdagisometry}. We claim that $\hat{\iota}_{\alpha}: \alpha^{**} \to \alpha$ is a modification. Indeed, by~\eqref{eq:doubledagger} and unitarity of $\{\iota_X\}$ we clearly have:
\begin{calign}\nonumber
\includegraphics[valign=c]{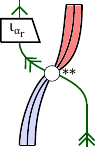}
~~=~~
\includegraphics[valign=c]{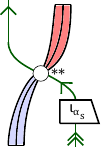}
\end{calign}
The proof that $\hat{\iota}$ is an invertible icon ${**}_{\Fun(\mathcal{C},\mathcal{D})} \to \id_{\Fun(\mathcal{C},\mathcal{D})}$, i.e. that the transformation is monoidal and natural, is given at the end of the proof of Proposition~\ref{prop:pivinduced}.

Finally, we must show that the duals of $\Fun_u(\mathcal{C},\mathcal{D})$ are dagger duals~\eqref{eq:pivdagcat}. This follows from the fact that the dagger of a modification is taken componentwise, and the cup and cap for each component come from the pivotal dagger structure in $\mathcal{D}$.
\end{proof}
\begin{corollary}
Let $\mathcal{C},\mathcal{D}$ be pivotal dagger 2-categories and let $\alpha: F_1 \to F_2$ be a UPT between pseudofunctors $\mathcal{C} \to \mathcal{D}$. Then the right dual $\alpha^*$ defined in Lemma~\ref{lem:pntduals} is equal to the left dual ${}^*\alpha$ defined in Lemma~\ref{lem:pntduals}.
\end{corollary}
\begin{proof}
For every 1-morphism $X: r \to s$ of $\mathcal{C}$ the right dual UPT satisfies the following equation with respect to the double right dual UPT:
\begin{calign}
\includegraphics[valign=c]{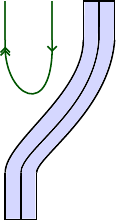}
~~=~~
\includegraphics[valign=c]{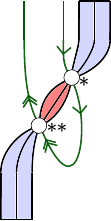}
~~=~~
\includegraphics[valign=c]{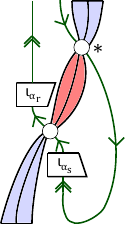}
\end{calign}
Postcomposing the leftmost and rightmost expressions by $\iota_{\alpha_r} \otimes \id_{\alpha_r^*} \otimes \id_{F_1(X)}$, we obtain the following pullthrough equation for the cup of the left duality:
\begin{calign}
\includegraphics[valign=c]{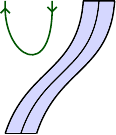}
~~=~~
\includegraphics[valign=c]{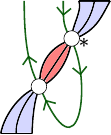}
\end{calign}
A similar pullthrough equation can be obtained for the cap of the left duality. It follows that $\alpha^*$ is a left dual of $\alpha$ with the same cup and cap as the chosen left dual of $\alpha$. We must therefore have $\alpha^*={}^*\alpha$ by Lemma~\ref{lem:relateduals} in $\Fun_u(\mathcal{C},\mathcal{D})$.
\end{proof}

\section{Morita theory for fibre functors on $\Rep(G)$}\label{sec:apps}

We finish by discussing an application of the results in this work. One reason for proving that $\Fun_u(\mathcal{C},\mathcal{D})$ is a pivotal dagger 2-category is that this provides an appropriate setting for Morita theory, which relates 1-morphisms out of an object $r$ to Frobenius monoids in its pivotal dagger category of endomorphisms $\End(r):=\Hom(r,r)$.

\begin{definition}\label{def:Frobenius}
A \emph{monoid} in a monoidal dagger category is an object $A$ with multiplication and unit morphisms, depicted as follows:%
\begin{calign}
\begin{tz}[zx,master]
\coordinate (A) at (0,0);
\draw (0.75,1) to (0.75,2);
\mult{A}{1.5}{1}
\end{tz}
&
\begin{tz}[zx,slave]
\coordinate (A) at (0.75,2);
\unit{A}{1}
\end{tz}
\\[0pt]\nonumber
m:A\otimes A \to A& u: \mathbbm{1} \to A 
\end{calign}\hspace{-0.2cm}
These morphisms satisfy the following associativity and unitality equations:
\begin{calign}\label{eq:assocandunitality}
\begin{tz}[zx]
\coordinate(A) at (0.25,0);
\draw (1,1) to [out=up, in=-135] (1.75,2);
\draw (1.75,2) to [out=-45, in=up] (3.25,0);
\draw (1.75,2) to (1.75,3);
\mult{A}{1.5}{1}
\node[zxvertex=\zxwhite,zxdown] at (1.75,2){};
\end{tz}
\quad = \quad
\begin{tz}[zx,xscale=-1]
\coordinate(A) at (0.25,0);
\draw (1,1) to [out=up, in=-135] (1.75,2);
\draw (1.75,2) to [out=-45, in=up] (3.25,0);
\draw (1.75,2) to (1.75,3);
\mult{A}{1.5}{1}
\node[zxvertex=\zxwhite,zxdown] at (1.75,2){};
\end{tz}
&
\begin{tz}[zx]
\coordinate (A) at (0,0);
\draw (0,-0.25) to (0,0);
\draw (0.75,1) to (0.75,2);
\mult{A}{1.5}{1}
\node[zxvertex=\zxwhite,zxdown] at (1.5,0){};
\end{tz}
\quad =\quad
\begin{tz}[zx]
\draw (0,0) to (0,2);
\end{tz}
\quad= \quad
\begin{tz}[zx,xscale=-1]
\coordinate (A) at (0,0);
\draw (0,-0.25) to (0,0);
\draw (0.75,1) to (0.75,2);
\mult{A}{1.5}{1}
\node[zxvertex=\zxwhite,zxdown] at (1.5,0){};
\end{tz}
\end{calign}
Analogously, a \textit{comonoid} is an object $A$ with a coassociative comultiplication $\delta: A \to A\otimes A$ and a counit $\epsilon:A\to \mathbbm{1}$. The dagger of an monoid $(A,m,u)$ is a comonoid $(A,m^{\dagger},u^{\dagger})$.  A monoid $(A,m,u)$ is called \textit{Frobenius} if the monoid and adjoint comonoid structures are related by the following \emph{Frobenius equation}:
\begin{calign}\label{eq:Frobenius}
\begin{tz}[zx]
\draw (0,0) to [out=up, in=-135] (0.75,2) to (0.75,3);
\draw (0.75,2) to [out=-45, in=135] (2.25,1);
\draw (2.25,0) to (2.25,1) to [out=45, in=down] (3,3);
\node[zxvertex=\zxwhite,zxup] at (2.25,1){};
\node[zxvertex=\zxwhite,zxdown] at (0.75,2){};
\end{tz}
\quad = \quad
\begin{tz}[zx]
\coordinate (A) at (0,0);
\coordinate (B) at (0,3);
\draw (0.75,1) to (0.75,2);
\mult{A}{1.5}{1}
\comult{B}{1.5}{1}
\end{tz}
\quad = \quad 
\begin{tz}[zx,xscale=-1]
\draw (0,0) to [out=up, in=-135] (0.75,2) to (0.75,3);
\draw (0.75,2) to [out=-45, in=135] (2.25,1);
\draw (2.25,0) to (2.25,1) to [out=45, in=down] (3,3);
\node[zxvertex=\zxwhite,zxup] at (2.25,1){};
\node[zxvertex=\zxwhite,zxdown] at (0.75,2){};
\end{tz}
\end{calign}
A Frobenius monoid is \emph{special} if the following equation is satisfied:
\begin{calign}
\begin{tz}[zx,every to/.style={out=up, in=down}]\draw (0,0) to (0,1) to [out=135] (-0.75,2) to [in=-135] (0,3) to (0,4);
\draw (0,1) to [out=45] (0.75,2) to [in=-45] (0,3);
\node[zxvertex=\zxwhite, zxup] at (0,1){};
\node[zxvertex=\zxwhite,zxdown] at (0,3){};\end{tz}
\quad = \quad 
\begin{tz}[zx]
\draw (0,0) to +(0,4);
\end{tz}
\end{calign}
\end{definition}
\noindent
Let $\mathcal{C}$ be a pivotal dagger 2-category. We say that $\mathcal{C}$ is \emph{$\mathbb{C}$-linear} if the $2$-morphism sets are complex vector spaces such that horizontal and vertical composition of 2-morphisms are bilinear maps and the dagger is an antilinear map. We assume additionally that, for any 2-morphism $f: X \to Y$, $f^{\dagger} \circ f = 0$ implies $f=0$. We say that an object $r$ of $\mathcal{C}$ is \emph{simple} if $\Hom(\id_r,\id_r) \cong \mathbb{C}$.

The variant of Morita theory consider is essentially as follows. Let $\mathcal{C}$ be a $\mathbb{C}$-linear pivotal dagger 2-category, let $s$ be a simple object, and let $X :r \to s$ be a 1-morphism. Let $d_X$ be the nonzero scalar such that $\dim_L(X) = d_X \id_s$. Observe that $\End(r)$ is a monoidal dagger category. Then, making use of the left duality in the pivotal dagger 2-category, the object $X \otimes X^*$ in $\End(r)$ acquires the structure of a special Frobenius monoid with the following multiplication and unit morphisms:
\begin{calign}\label{eq:frobeniusmorita}
\frac{1}{\sqrt{d_X}}~~
\includegraphics[valign=c]{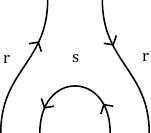}
&&
\sqrt{d_X}~~
\includegraphics[valign=c]{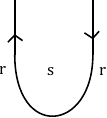}
\end{calign}
We will see that 1-morphisms from $r$ to simple objects can in fact be classified in terms of relations between their corresponding special Frobenius monoids. 

We are here particularly interested in the pivotal dagger 2-category $\Fun_u(\Rep(G),\Hilb)$, where $\Rep(G)$ is the pivotal dagger category of continuous finite-dimensional unitary representations of a compact quantum group $G$ and $\Hilb$ is the category of Hilbert spaces and linear maps. We restrict to $\mathbb{C}$-linear unitary monoidal functors, which we call \emph{fibre functors}. 
It is not important for our purposes here to discuss the definition of the category $\Rep(G)$ (see e.g.~\cite[\S{}2.3.2]{Verdon2020} for this). All that matters here is that $\Rep(G)$ is a pivotal dagger category with a privileged \emph{canonical fibre functor} $F: \Rep(G) \to \Hilb$.

In this case, since $\Rep(G)$ and $\Hilb$ are one-object 2-categories, we obtain a simpler description of UPTs and modifications. In particular:

\begin{itemize}
\item Let $F_1,F_2$ be fibre functors on $\Rep(G)$.
A unitary pseudonatural transformation $(\alpha,H): F_1 \to F_2$ is defined by the following data:
\begin{itemize}
\item An Hilbert space $H$ (drawn as a green wire).
\item For every object $X$ of $\Rep(G)$, a unitary $\alpha_X: F_1(X) \otimes H \to H \otimes F_2(X)$ (drawn as a white vertex):
\begin{calign}
\includegraphics[scale=1.2,valign=c]{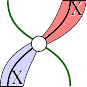}
\end{calign}
\end{itemize}
These unitaries must obey the naturality and monoidality conditions~(\ref{eq:pntnat}-\ref{eq:pntmonunit}). We call $\dim(H)$ the \emph{dimension} of the UPT.

\item Let $(\alpha,H), (\beta,H'): F_1 \to F_2$ be UPTs. (We colour the $H$-wire green and the $H'$-wire orange.) A modification $f: \alpha \to \beta$  is a linear map $f: H \to H'$ satisfying the following equation for all unitaries $\{\alpha_X,\beta_X\}$:
\begin{calign}\label{eq:uptmod}
\includegraphics[scale=1,valign=c]{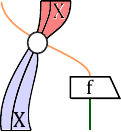}
~~=~~
\includegraphics[scale=1,valign=c]{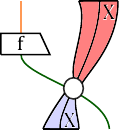}
\end{calign}
\end{itemize}
It is clear that $\Fun_u(\Rep(G),\Hilb)$ is $\mathbb{C}$-linear. Moreover, every object of $\Fun(\Rep(G),\Hilb)$ is simple.

In~\cite[\S{}3]{Verdon2020} we characterised the category $\End(F)$ of unitary pseudonatural transformations and modifications from the canonical fibre functor to itself, showing that it is isomorphic to the category $\Rep(A_G)$ of finite-dimensional $*$-representations of the \emph{compact quantum group algebra} $A_G$ associated to the compact quantum group $G$. 
Morita theory will therefore allow us to classify fibre functors accessible by a UPT from the canonical fibre functor, and UPTs from the canonical fibre functor, in terms of special Frobenius monoids in the category $\Rep(A_G)$.

\subsection{Classification of UPTs from the canonical fibre functor}
We begin with a technical definition.\begin{definition}
We say that a dagger 2-category has \emph{split dagger idempotents} if, for any 1-morphism $X: r \to s$ and any 2-morphism $\alpha: X \to X$ such that $\alpha=\alpha^{\dagger}=\alpha^2$ (we call such 2-morphisms \emph{dagger idempotent}), there exists a 1-morphism $V: r \to s$ and an isometry $\iota: V \to X$ such that $\iota \circ \iota^{\dagger} = \alpha$.
\end{definition}
\begin{lemma}
The category $\Fun_u(\mathcal{C},\mathcal{D})$ has split dagger idempotents if $\mathcal{D}$ has split dagger idempotents. 
\end{lemma}
\begin{proof}
Let $\alpha: F_1 \to F_2$ be a UPT and let $f: \alpha \to \alpha$ be a dagger idempotent modification. Since for each object $r$ of $\mathcal{C}$ the component $f_r:\alpha_r \to \alpha_r$ is itself a dagger idempotent in $\mathcal{D}$, there exist objects $I_r$ of $D$ and isometries $\iota_{f,r}: I_r \to \alpha_r$ such that: 
\begin{align}\label{eq:splittingconjisomeqs}
\iota_{f,r}^{\dagger} \circ \iota_{f,r}   = \id_{I_r} &&  \iota_{f,r} \circ \iota_{f,r}^{\dagger} = f_r
\end{align}
Now we define a new UPT $\alpha^{\iota_f}$ whose components $\{\alpha^{\iota_f}_X\}$ are given as follows:
\begin{calign}
\includegraphics[valign=c]{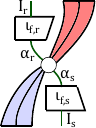}
\end{calign}
It is clear that this is a UPT and that $\iota_f$, with components defined as $(\iota_{f})_r = \iota_{f,r}$, is a modification $\alpha^{\iota_f} \to \alpha$ satisfying $\iota_f^{\dagger} \circ \iota_f = \id_{\alpha}$ and $\iota_f \circ \iota_f^{\dagger}= f$.
\end{proof}
\noindent
It immediately follows that $\Fun_u(\Rep(G),\Hilb)$ has split dagger idempotents, since $\Hilb$ does. In order to classify UPTs from the canonical fibre functor we will need a notion of equivalence of 1-morphisms.
\begin{definition}
Let $r,s,t$ be objects in a dagger 2-category $\mathcal{C}$. We say that two 1-morphisms $X: r \to s$ and $Y: r \to t$ are \emph{equivalent} if there exists a dagger equivalence $E: t \to s$ and a unitary 2-morphism $\tau: X \to Y \otimes E$.
\end{definition}
\noindent
In $\Fun_u(\Rep(G),\Hilb)$ equivalence of UPTs can be put in more familiar terms.
\begin{lemma}
Two UPTs $\alpha_1: F \to F_1$ and $\alpha_2: F \to F_2$ are equivalent in $\Fun_u(\Rep(G),\Hilb)$ if and only if there exists a unitary monoidal natural isomorphism $E: F_2 \to F_1$ and a unitary modification $\tau: \alpha_1 \to \alpha_2 \otimes E$.
\end{lemma}
\begin{proof}
Suppose that there is an equivalence $\alpha_1 \cong \alpha_2$. 
Let $[\tilde{E}:t \to s, \tilde{E}^{-1},\eta,\epsilon]$ be the data of the dagger equivalence $F_2 \to F_1$, and let $\tilde{\tau}: \alpha_1 \to \alpha_2 \otimes \tilde{E}$ be the unitary modification.

We first observe that $\eta$ is a unitary modification $\id_{F_1} \to \tilde{E}^{-1} \otimes \tilde{E}$. Considering underlying Hilbert spaces this yields a unitary map $\mathbb{C} \to H_{\tilde{E}^{-1}} \otimes H_{\tilde{E}}$, which implies that both these Hilbert spaces are one-dimensional. Therefore there is a unitary isomorphism $\omega: H_{\tilde{E}} \to \mathbb{C}$. Conjugating $\tilde{E}$ by this unitary isomorphism we obtain a unitary monoidal natural isomorphism $E: F_2 \to F_1$ (i.e. a UPT whose underlying Hilbert space is $\mathbb{C}$). Then $\tau:= (\id_{\alpha_2} \otimes \omega) \circ \tilde{\tau}$ is a unitary modification $\alpha_1 \to \alpha_2 \otimes E$.

In the other direction, a unitary monoidal natural isomorphism is a dagger equivalence: the weak inverse is the actual inverse, and the unitary modifications witnessing the equivalence are trivial.
\end{proof}
\noindent
We now define a corresponding equivalence relation for Frobenius monoids. 
\begin{definition}
Let $A,B$ be Frobenius monoids in a monoidal dagger category. We say that a morphism $f: A \to B$ is a \emph{$*$-homomorphism} precisely when it satisfies the following equations:
\begin{calign}\label{eq:homo}
\begin{tz}[zx, master, every to/.style={out=up, in=down},yscale=-1]
\draw (0,0) to (0,2) to [out=135] (-0.75,3);
\draw (0,2) to [out=45] (0.75, 3);
\node[zxnode=\zxwhite] at (0,1) {$f$};
\node[zxvertex=\zxwhite, zxdown] at (0,2) {};
\end{tz}
=
\begin{tz}[zx, every to/.style={out=up, in=down},yscale=-1]
\draw (0,0) to (0,0.75) to [out=135] (-0.75,1.75) to (-0.75,3);
\draw (0,0.75) to [out=45] (0.75, 1.75) to +(0,1.25);
\node[zxnode=\zxwhite] at (-0.75,2) {$f$};
\node[zxnode=\zxwhite] at (0.75,2) {$f$};
\node[zxvertex=\zxwhite, zxdown] at (0,0.75) {};
\end{tz}
&
\begin{tz}[zx,slave, every to/.style={out=up, in=down},yscale=-1]
\draw (0,0) to (0,2) ;
\node[zxnode=\zxwhite] at (0,1) {$f$};
\node[zxvertex=\zxwhite, zxup] at (0,2) {};
\end{tz}
=
\begin{tz}[zx,slave, every to/.style={out=up, in=down},yscale=-1]
\draw (0,0) to (0,0.75) ;
\node[zxvertex=\zxwhite, zxup] at (0,0.75) {};
\end{tz}
&
\begin{tz}[zx,slave, every to/.style={out=up, in=down},scale=-1]
\draw (0,0) to (0,3);
\node[zxnode=\zxwhite] at (0,1.5) {$f^\dagger$};
\end{tz}
=~~
\begin{tz}[zx,slave,every to/.style={out=up, in=down},scale=-1,xscale=1]
\draw (0,1.5) to (0,2) to [in=left] node[pos=1] (r){} (0.5,2.5) to [out=right, in=up] (1,2)  to [out=down, in=up] (1,0);
\draw (-1,3) to [out=down,in=up] (-1,1) to [out=down, in=left] node[pos=1] (l){} (-0.5,0.5) to [out=right, in=down] (0,1) to (0,1.5);
\draw (0.5,2.5) to[in=down] (0.5,3); 
\node[zxvertex=\zxwhite] at (0.5,3){};
\draw (-0.5,0) to[in=down] (-0.5,0.5);
\node[zxvertex=\zxwhite] at (-.5,0){};
\node[zxnode=\zxwhite] at (0,1.5) {$f$};
\node[zxvertex=\zxwhite] at (l.center){};
\node[zxvertex=\zxwhite] at (r.center){};
\end{tz}
\end{calign}
We call a unitary $*$-homomorphism a \emph{unitary $*$-isomorphism}. It is easy to check that a unitary $*$-isomorphism also obeys the following \emph{$*$-cohomomorphism} equations:
\begin{calign}
\label{eq:cohomo}
\begin{tz}[zx, master, every to/.style={out=up, in=down}]
\draw (0,0) to (0,2) to [out=135] (-0.75,3);
\draw (0,2) to [out=45] (0.75, 3);
\node[zxnode=\zxwhite] at (0,1) {$f$};
\node[zxvertex=\zxwhite, zxup] at (0,2) {};
\end{tz}
=
\begin{tz}[zx, every to/.style={out=up, in=down}]
\draw (0,0) to (0,0.75) to [out=135] (-0.75,1.75) to (-0.75,3);
\draw (0,0.75) to [out=45] (0.75, 1.75) to +(0,1.25);
\node[zxnode=\zxwhite] at (-0.75,2) {$f$};
\node[zxnode=\zxwhite] at (0.75,2) {$f$};
\node[zxvertex=\zxwhite, zxup] at (0,0.75) {};
\end{tz}
&
\begin{tz}[zx,slave, every to/.style={out=up, in=down}]
\draw (0,0) to (0,2) ;
\node[zxnode=\zxwhite] at (0,1) {$f$};
\node[zxvertex=\zxwhite, zxup] at (0,2) {};
\end{tz}
=
\begin{tz}[zx,slave, every to/.style={out=up, in=down}]
\draw (0,0) to (0,0.75) ;
\node[zxvertex=\zxwhite, zxup] at (0,0.75) {};
\end{tz}
&
\begin{tz}[zx,slave, every to/.style={out=up, in=down}]
\draw (0,0) to (0,3);
\node[zxnode=\zxwhite] at (0,1.5) {$f^\dagger$};
\end{tz}
=~~
\begin{tz}[zx,slave,every to/.style={out=up, in=down}]
\draw (0,1.5) to (0,2) to [in=left] node[pos=1] (r){} (0.5,2.5) to [out=right, in=up] (1,2)  to [out=down, in=up] (1,0);
\draw (-1,3) to [out=down,in=up] (-1,1) to [out=down, in=left] node[pos=1] (l){} (-0.5,0.5) to [out=right, in=down] (0,1) to (0,1.5);
\draw (0.5,2.5) to[in=down] (0.5,3); 
\node[zxvertex=\zxwhite] at (0.5,3){};
\draw (-0.5,0) to[in=down] (-0.5,0.5);
\node[zxvertex=\zxwhite] at (-.5,0){};
\node[zxnode=\zxwhite] at (0,1.5) {$f$};
\node[zxvertex=\zxwhite] at (l.center){};
\node[zxvertex=\zxwhite] at (r.center){};
\end{tz}
\end{calign}
\end{definition}
\begin{theorem}\label{thm:starisoclass}
Let $\mathcal{C}$ be a $\mathbb{C}$-linear pivotal dagger 2-category with split dagger idempotents. Let $s,t$ be simple objects, and let $X:r \to s$ and $Y: r \to t$ be 1-morphisms. Then $X$ and $Y$ are equivalent in $\mathcal{C}$ if and only if the special Frobenius monoids $X \otimes X^*$ and $Y \otimes Y^*$ in $\End(r)$ are unitarily $*$-isomorphic.
\end{theorem}
\begin{proof}
Suppose that $X$ and $Y$ are equivalent by some dagger equivalence $[E,E^{-1},\eta,\epsilon]$ and unitary 2-morphism $\tau: X \to Y \otimes E$. WLOG we may take $[E^{-1},\eta,\epsilon]$ to be a right dual for $E$. We will show that $X \otimes X^*$ and $Y \otimes Y^*$ are unitarily $*$-isomorphic. 

We first consider the relationship between the right dual $[E^{-1},\eta,\epsilon]$ and the chosen right dual for $E$ in the pivotal dagger 2-category $\mathcal{C}$. Let $u: E^* \to E^{-1}$ be the isomorphism relating the right duals $E^*$ and $E^{-1}$ by Lemma~\ref{lem:relateduals}:
\begin{calign}
\includegraphics[scale=.8,valign=c]{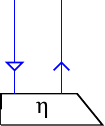}
~~=~~
\includegraphics[scale=.8,valign=c]{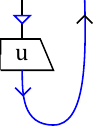}
&&
\includegraphics[scale=.8,valign=c]{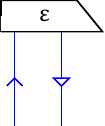}
~~=~~
\includegraphics[scale=.8,valign=c]{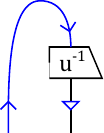}
\end{calign}
(Here and throughout we draw the equivalence $E$ and its duals with a blue wire, and the $E^{-1}$ wire with a triangular arrow.) 
Let $d_E$ be the scalar such that $\dim_R(E) = \frac{1}{d_E} \id_t$.
We first observe that 
\begin{equation}
u^{\dagger} = \frac{1}{d_E} u^{-1},
\end{equation}
which can be seen by the following equalities:
\begin{calign}\nonumber
\includegraphics[scale=.8,valign=c]{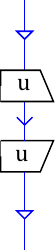}
~~=~~
\includegraphics[scale=.8,valign=c]{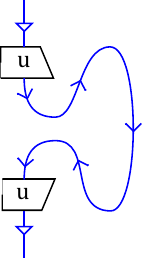}
~~=~~
\includegraphics[scale=.8,valign=c]{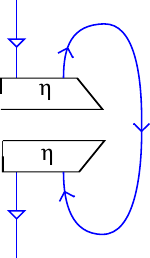}
~~=~~
\includegraphics[scale=.8,valign=c]{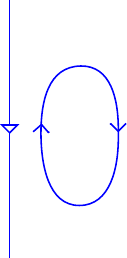}
\end{calign}
We can therefore make the following further observation:
\begin{calign}\label{eq:capcup}
\includegraphics[scale=.8,valign=c]{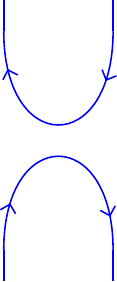}
~~=~~
\includegraphics[scale=.8,valign=c]{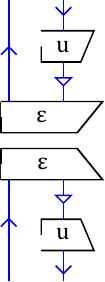}
~~=~~
\frac{1}{d_E}~
\includegraphics[scale=.8,valign=c]{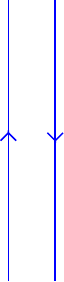}
\end{calign}
We also note that 
\begin{equation}\dim_L(E) = d_E \id_s,
\end{equation}
which is seen as follows:
\begin{calign}\nonumber
\id_{s} 
~~=~~ 
\includegraphics[scale=.8,valign=c]{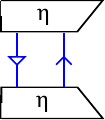}
~~=~~
\includegraphics[scale=.8,valign=c]{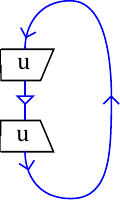}
~~=~~
\frac{1}{d_E}~
\includegraphics[scale=.8,valign=c]{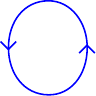}
~~=~~
\frac{1}{d_E}\dim_L(E)
\end{calign}
Finally we consider the relationship between $d_E$ and $d_X,d_Y$:
\begin{calign}\nonumber
d_X \id_s
~~=~~
\includegraphics[scale=1,valign=c]{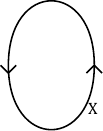}
~~=~~
\includegraphics[scale=1,valign=c]{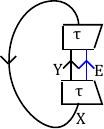}
~~=~~
\includegraphics[scale=1,valign=c]{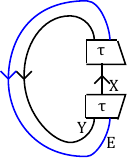}
~~=~~
\includegraphics[scale=1,valign=c]{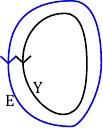}
~~=~~d_Y d_E \id_s
\end{calign}
Here the second and fourth equalities are by unitarity of $\tau$, and the third equality is by pulling $\tau$ around the cup and cap of the duality.
It follows that:
\begin{equation}\label{eq:dedef}
d_E = \frac{d_X}{d_Y}
\end{equation}
Now we can define our unitary $*$-isomorphism $X \otimes X^* \to Y \otimes Y^*$. Consider the following 2-morphism:
\begin{calign}
\sqrt{\frac{d_X}{d_Y}}~~
\includegraphics[scale=1,valign=c]{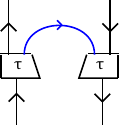}
\end{calign}
We show that this 2-morphism is a unitary $*$-isomorphism. For unitarity:
\begin{calign}
\frac{d_X}{d_Y}~~
\includegraphics[scale=1,valign=c]{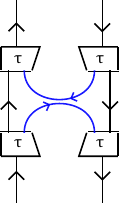}
~~=~~
\includegraphics[scale=1,valign=c]{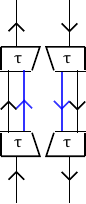}
~~=~~
\includegraphics[scale=1,valign=c]{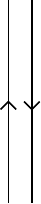}
\end{calign}
Here the first equality is by~\eqref{eq:capcup} and the second is by unitarity of $\tau$.
\begin{calign}
\frac{d_X}{d_Y}~~
\includegraphics[scale=1,valign=c]{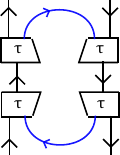}
~~=~~
\frac{d_X}{d_Y}~~
\includegraphics[scale=1,valign=c]{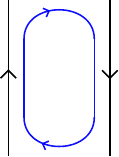}
~~=~~
\includegraphics[scale=1,valign=c]{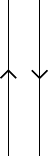}
\end{calign}
Here the first equality is by unitarity of $\tau$ and the second is by~\eqref{eq:dedef}.

For the first $*$-homomorphism condition:
\begin{calign}\nonumber
\frac{d_X}{d_Y} \cdot \frac{1}{\sqrt{d_Y}}~~
\includegraphics[scale=1,valign=c]{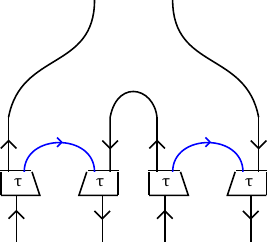}
~~=~~
\frac{d_X}{d_Y} \cdot \frac{1}{\sqrt{d_Y}}~~\includegraphics[scale=1,valign=c]{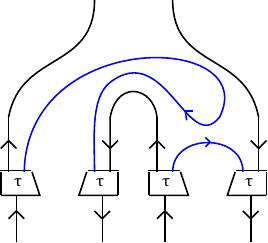}
\\\nonumber
=~~
\frac{1}{\sqrt{d_Y}}
\includegraphics[scale=1,valign=c]{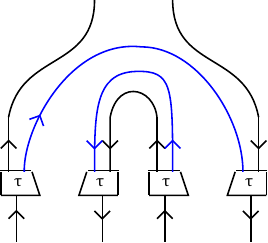}
~~=~~
\frac{1}{\sqrt{d_Y}}
\includegraphics[scale=1,valign=c]{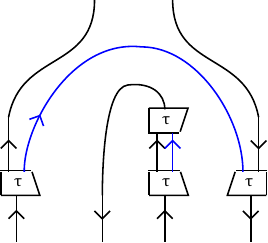}
\\=~~
\sqrt{\frac{d_X}{d_Y}} \cdot \frac{1}{\sqrt{d_X}}~~
\includegraphics[scale=1,valign=c]{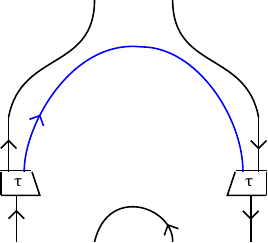}
\end{calign}
Here the second equality is by~\eqref{eq:capcup} and the fourth equality is by unitarity of $\tau$.

For the second $*$-homomorphism condition:
\begin{calign}
\sqrt{\frac{d_X}{d_Y}} \cdot \sqrt{d_X}~~
\includegraphics[scale=1,valign=c]{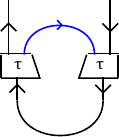}
~~=~~
\frac{d_X}{\sqrt{d_Y}}
\includegraphics[scale=1,valign=c]{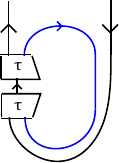}
~~=~~
\frac{d_X}{\sqrt{d_Y}}
\includegraphics[scale=1,valign=c]{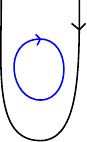}
~~=~~
\sqrt{d_Y}
\includegraphics[scale=1,valign=c]{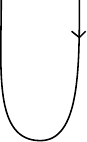}
\end{calign}
Here the second equality is by unitarity of $\tau$ and the third equality is by definition of $d_E$.

The third $*$-homomorphism condition is implied by unitarity and the first two $*$-homomorphism conditions.

One direction is therefore proved. For the other direction, let $f: X \otimes X^* \to Y \otimes Y^*$ be a unitary $*$-isomorphism. We will now construct a dagger equivalence $E: t \to s$ and a unitary 2-morphism $\tau: X \to Y \otimes E$.

We first observe that the following modification $\tilde{f}: Y^* \otimes X \to Y^* \otimes X$ is a dagger idempotent:
\begin{calign}
\frac{1}{\sqrt{d_Xd_Y}}
\includegraphics[scale=1,valign=c]{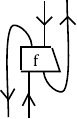} 
\end{calign}
Indeed, we have the following equations for dagger idempotency. For idempotency:
\begin{calign}
\frac{1}{d_Xd_Y}
\includegraphics[scale=1,valign=c]{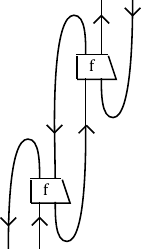}
~~=~~
\frac{1}{\sqrt{d_Y}(d_X)^{3/2}}
\includegraphics[scale=1,valign=c]{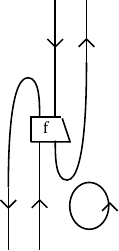}
~~=~~
\frac{1}{\sqrt{d_Xd_Y}}
\includegraphics[scale=1,valign=c]{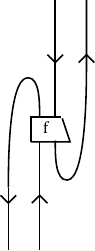}
\end{calign}
Here the first equality is by the first $*$-homomorphism condition~\eqref{eq:homo}. To see that the idempotent is dagger:
\begin{calign}
\frac{1}{\sqrt{d_Xd_Y}}
\includegraphics[scale=1,valign=c]{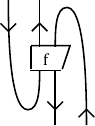}
~~=~~
\frac{1}{\sqrt{d_Xd_Y}}
\includegraphics[scale=1,valign=c]{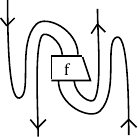}
~~=~~
\frac{1}{\sqrt{d_Xd_Y}}
\includegraphics[scale=1,valign=c]{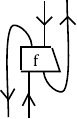}
\end{calign}
Here the first equality is by the third $*$-cohomomorphism condition~\eqref{eq:cohomo}.

Since dagger idempotents split, we obtain a new 1-morphism $E: t \to s$ and an isometry $\tilde{\tau}: E \to Y^* \otimes X$ satisfying $\tilde{\tau} \circ \tilde{\tau}^{\dagger} = \tilde{f}$, i.e.:
\begin{calign}
\includegraphics[scale=1,valign=c]{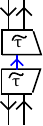}
~~=~~
\frac{1}{\sqrt{d_Xd_Y}}
\includegraphics[scale=1,valign=c]{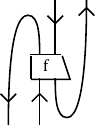}
\end{calign}
We will first show that $E$ is a dagger equivalence. Indeed, we observe that 
\begin{equation}\label{eq:dimre}
\dim_R(E) = \frac{d_Y}{d_X} \id_t
\end{equation}
by the following equalities:
\begin{calign}
\includegraphics[scale=1,valign=c]{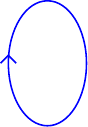}
~~=~~
\includegraphics[scale=1,valign=c]{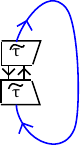}
~~=~~
\includegraphics[scale=1,valign=c]{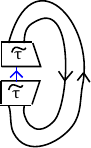}
~~=~~
\frac{1}{\sqrt{d_X d_Y}}
\includegraphics[scale=1,valign=c]{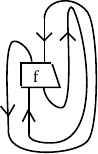}
~~=~~
\frac{1}{d_X}
\includegraphics[scale=1,valign=c]{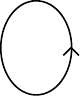}
\end{calign}
Here the first equality is by the fact that $\tilde{\tau}$ is an isometry, the second equality is by sliding $\tilde{\tau}$ around the cup and cap, the third equality is by $\tilde{\tau} \circ \tilde{\tau}^{\dagger} = \tilde{f}$, and the fourth equality is by the second $*$-homomorphism condition~\eqref{eq:homo}.
Likewise, we can show 
\begin{equation}\label{eq:dimle}
\dim_L(E) = \frac{d_X}{d_Y} \id_s;
\end{equation}
for this we use the same technique with the second $*$-cohomomorphism condition~\eqref{eq:cohomo}. 

We therefore propose that $E^*$ is a weak inverse for $E$, with the following 2-morphisms witnessing the equivalence:
\begin{calign}\label{eq:equivwitnesses}
\sqrt{\frac{d_Y}{d_X}}~~
\includegraphics[scale=1,valign=c]{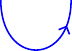}
&&
\sqrt{\frac{d_X}{d_Y}}~~
\includegraphics[scale=1,valign=c]{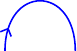}
\end{calign}
The equations~\eqref{eq:dimle} and~\eqref{eq:dimre} show that the 2-morphisms~\eqref{eq:equivwitnesses} are an isometry and a coisometry respectively. For unitarity we must show that they are also a coisometry and an isometry respectively.

For this we first observe the following decomposition of the unitary $*$-isomorphism $f$ in terms of the isometry $\tilde{\tau}$, which follows straightforwardly from the definition of $\tilde{f}$ and $\tilde{\tau}$:
\begin{calign}\label{eq:ffromtau}
\includegraphics[scale=1,valign=c]{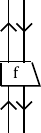}
~~=~~
\includegraphics[scale=1,valign=c]{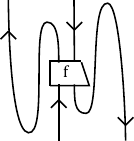}
~~=~~
\sqrt{d_X d_Y}
\includegraphics[scale=1,valign=c]{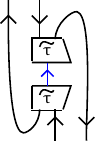}
~~=~~
\sqrt{d_X d_Y}
\includegraphics[scale=1,valign=c]{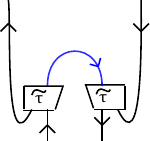}
\end{calign}
It will also be useful to note the following expression of $f^{\dagger}$ in terms of $\tilde{\tau}$ for later:
\begin{calign}\label{eq:fdaggerfromtau}
\includegraphics[scale=1,valign=c]{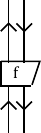}
~~=~~
\includegraphics[scale=1,valign=c]{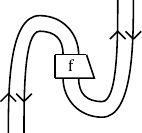}
~~=~~
\sqrt{d_X d_Y}
\includegraphics[scale=1,valign=c]{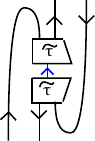}
~~=~~
\sqrt{d_X d_Y}
\includegraphics[scale=1,valign=c]{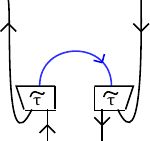}
\end{calign}
Here the second equality was by the third $*$-cohomomorphism equation~\eqref{eq:cohomo}.

Using~\eqref{eq:ffromtau}, we now consider what the first $*$-homomorphism~\eqref{eq:homo} and $*$-cohomomorphism~\eqref{eq:cohomo} equations tell us about $\tilde{\tau}$. We begin with the first $*$-homomorphism equation:
\begin{calign}\nonumber
d_X \sqrt{d_Y}
\includegraphics[scale=1,valign=c]{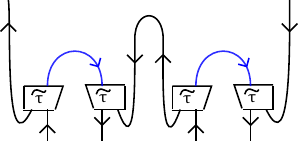}
~~=~~
\sqrt{d_Y}
\includegraphics[scale=1,valign=c]{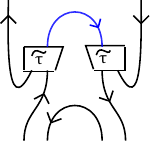}
\\\nonumber
\Rightarrow~~
d_X~
\includegraphics[scale=1,valign=c]{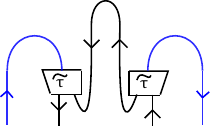}
~~=~~
\includegraphics[scale=1,valign=c]{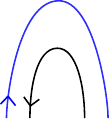}
\\\label{eq:taustarhom1}
\Leftrightarrow~~
d_X
\includegraphics[scale=1,valign=c]{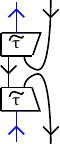}
~~=~~
\includegraphics[scale=1,valign=c]{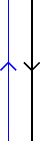}
\qquad \qquad \qquad
\Leftrightarrow~~
d_X~
\includegraphics[scale=1,valign=c]{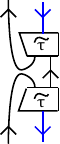}
~~=~~
\includegraphics[scale=1,valign=c]{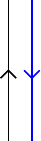}
\end{calign}
Here for the first implication we bent the top left and top right legs down and precomposed with $\tilde{\tau}$ on the left and $\tilde{\tau}^*$ on the right, using the fact that $\tilde{\tau}$ is an isometry. For the second implication we bent the two rightmost legs upwards. For the third implication we took the transpose.

We now consider the first $*$-cohomomorphism equation (the derivation of these implications is precisely as before):
\begin{calign}\nonumber
\sqrt{d_X} d_Y
\includegraphics[scale=1,valign=c]{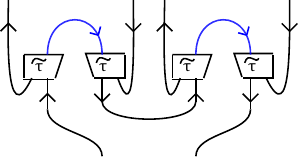}
~~=~~
\sqrt{d_X}
\includegraphics[scale=1,valign=c]{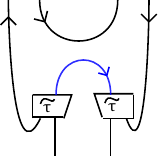}
\\\nonumber
\Rightarrow~~
d_Y~
\includegraphics[scale=1,valign=c]{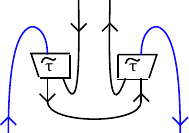}
~~=~~
\includegraphics[scale=1,valign=c]{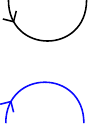}
\\ \label{eq:taustarcohom1}
\Leftrightarrow~~
d_Y~
\includegraphics[scale=1,valign=c]{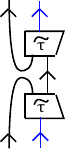}
~~=~~
\includegraphics[scale=1,valign=c]{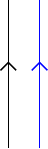}
\qquad \qquad \qquad
\Leftrightarrow~~
d_Y~
\includegraphics[scale=1,valign=c]{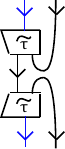}
~~=~~
\includegraphics[scale=1,valign=c]{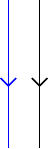}
\end{calign}
These equations are all we need to show that the 2-morphisms~\eqref{eq:equivwitnesses} are unitary. Indeed, we show that the first is a coisometry:
\begin{calign}\nonumber
\frac{d_Y}{d_X}~~
\includegraphics[scale=1,valign=c]{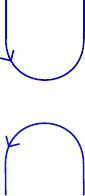}
~~=~~
\frac{d_Y}{(d_X)^3}~
\includegraphics[scale=1,valign=c]{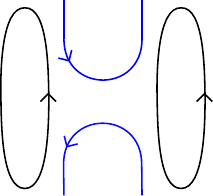}
\\\nonumber
=~~
d_X d_Y
\includegraphics[scale=1,valign=c]{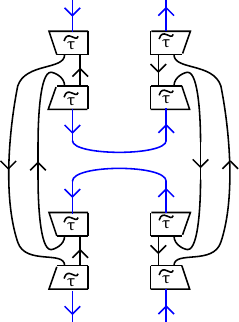}
~~=~~
\includegraphics[scale=1,valign=c]{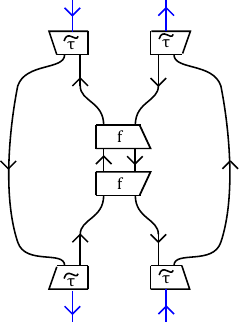}
\\
=~~
\includegraphics[scale=1,valign=c]{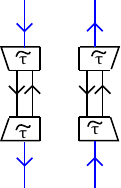}
~~=~~
\includegraphics[scale=1,valign=c]{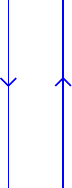}
\end{calign}
Here the second equality is by~\eqref{eq:taustarhom1}; the third equality is by~\eqref{eq:fdaggerfromtau}; the fourth equality is by unitarity of $f$; and the last equality follows since $\tilde{\tau}$ is an isometry.

We similarly show that the second 2-morphism of~\eqref{eq:equivwitnesses} is an isometry:
\begin{calign}\nonumber
\frac{d_X}{d_Y}~~
\includegraphics[scale=1,valign=c]{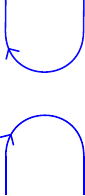}
~~=~~
\frac{d_X}{(d_Y)^3}~
\includegraphics[scale=1,valign=c]{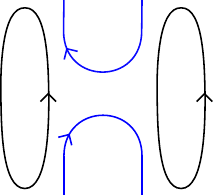}
\\\nonumber
=~~
d_X d_Y
\includegraphics[scale=1,valign=c]{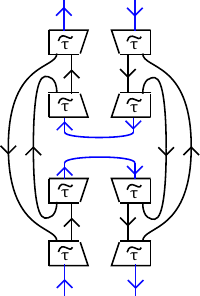}
~~=~~
\includegraphics[scale=1,valign=c]{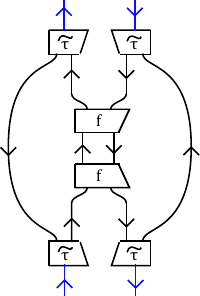}
\\
=~~
\includegraphics[scale=1,valign=c]{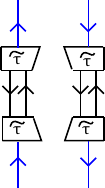}
~~=~~
\includegraphics[scale=1,valign=c]{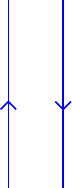}
\end{calign}
Here the second equality is by~\eqref{eq:taustarcohom1}; the third equality is by~\eqref{eq:ffromtau}; the fourth equality is by unitarity of $f$; and the last equality follows since $\tilde{\tau}$ is an isometry. 

We have therefore shown that $E$ is a dagger equivalence. Lastly, we need to define a unitary 2-morphism $\tau: X \to Y \otimes E$. We define $\tau$ to be the following 2-morphism:
\begin{calign}
\sqrt{d_Y}
\includegraphics[scale=1,valign=c]{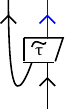}
\end{calign}
We need to show that $\tau$ is unitary. We already saw that it is a coisometry~\eqref{eq:taustarcohom1}. To show that it is an isometry we consider the second $*$-cohomomorphism equation~\eqref{eq:cohomo}:
\begin{calign}\nonumber
\sqrt{d_X}d_Y~~
\includegraphics[scale=1,valign=c]{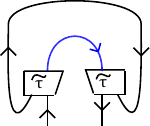}
~~=~~
\sqrt{d_X}~~
\includegraphics[scale=1,valign=c]{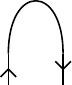}
\\
\Leftrightarrow~~
d_Y~~
\includegraphics[scale=1,valign=c]{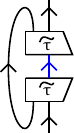}
~~=~~
\includegraphics[scale=1,valign=c]{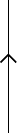}
\end{calign}
Here the implication is by bending the bottom right leg upwards.
The 2-morphism $\tau$ is therefore unitary and the result follows.  
\end{proof}
\noindent
We are almost ready to classify UPTs from the canonical fibre functor $F$. To classify UPTs in terms of special Frobenius monoids, we need some intrinsic characterisation of those special Frobenius monoids in $\End(F)$ which are \emph{split}: that is, which arise as $\alpha \otimes \alpha^*$ for some UPT $\alpha$ whose source is $F$. In~\cite[Def. 4.10]{Verdon2020} (c.f.~\cite[Def. 3.2]{Musto2019}) we define the notion of a \emph{simple Frobenius monoid} in $\End(F)$. For any simple Frobenius monoid $A$, we construct a fibre functor $F'$ and a UPT $\alpha: F \to F'$ such that $A \cong \alpha \otimes \alpha^*$. In the other direction, every special Frobenius algebra $\alpha \otimes \alpha^*$ is a simple Frobenius monoid.

By Theorem~\ref{thm:starisoclass} we therefore obtain the following classification. 
\begin{corollary}
Let $G$ be a compact quantum group and let $F: \Rep(G) \to \Hilb$ be the canonical fibre functor. There is a bijective correspondence between:
\begin{itemize}
\item Unitary $*$-isomorphism classes of simple Frobenius monoids in $\End(F) \cong \Rep(A_G)$.
\item Equivalence classes of UPTs whose source is the canonical fibre functor $F$.
\end{itemize}
\end{corollary}

\subsection{Classification of fibre functors}
We have classified equivalence classes of UPTs from the canonical fibre functor. We now classify the fibre functors $F'$ accessible from the canonical fibre functor $F$, i.e such that there exists a UPT $\alpha: F \to F'$.

We first observe another perspective on the special Frobenius monoid~\eqref{eq:frobeniusmorita}. 
\begin{definition}
Let $X: r \to s$ be a 1-morphism in a dagger 2-category. We say that $X$ is \emph{special} if it has a right dual $[X^*,\eta_s,\epsilon_s]$ satisfying the following equation:
\begin{calign}
\includegraphics[scale=1,valign=c]{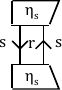}
~~=~~
\includegraphics[scale=1,valign=c]{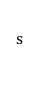}
\end{calign} 
\end{definition}
\begin{lemma}\label{lem:1morphspecial}
In a $\mathbb{C}$-linear pivotal dagger 2-category, all 1-morphisms into a simple object are special. 
\end{lemma}
\begin{proof}
Let $X: r \to s$ be a 1-morphism into a simple object, and let $[\alpha^*,\eta,\epsilon]$ be its chosen right dual. Let $d_X$ be the nonzero scalar such that $\dim_L(X) = d_X \id_s$. Now we normalise the cup and cap 2-morphisms:
\begin{align*}
\tilde{\eta} := \frac{1}{\sqrt{d_X}}\eta && \tilde{\epsilon} := \sqrt{d_X}\epsilon 
\end{align*}
Clearly the snake equations will still be obeyed.
\end{proof}
\noindent
A 1-morphism $X: r \to s$ in a dagger 2-category with a special right dual $[X^*,\tilde{\eta},\tilde{\epsilon}]$ induces a special Frobenius monoid on the object $X \otimes X^*$ in $\End(r)$, with multiplication and unit defined as follows:
\begin{calign}\label{eq:specialdualfrobmon}
\includegraphics[scale=.7,valign=c]{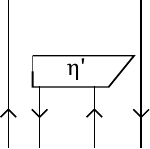}
&&
\includegraphics[scale=.7,valign=c]{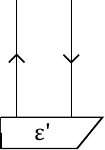}
\end{calign}
We observe that, in a pivotal dagger $2$-category, when the special dual of a 1-morphism is defined as in Lemma~\ref{lem:1morphspecial} then~\eqref{eq:specialdualfrobmon} is precisely the the special Frobenius monoid of~\eqref{eq:frobeniusmorita}.

For our classification we use the notion of \emph{Morita equivalence} of special Frobenius monoids.
\def\d{0.5}
\def\h{2.25}
\def\inang{-45}
\begin{definition} \label{def:daggerbimodule}Let $A$ and $B$ be special Frobenius monoids in a monoidal dagger category. An $A{-}B$-\textit{dagger bimodule} is an object $M$ together with an morphism $\rho:A\otimes M\otimes B \to M$ fulfilling the following equations:
\begin{calign}\label{eq:bimodule}
\begin{tz}[zx,every to/.style={out=up, in=down}]
\draw (0,0) to (0,3);
\draw (-\d-1.5,0) to [in=-135] (-\d-0.75,1) to[in=180-\inang] (0,\h);
\draw (-\d,0) to [in=-45] (-\d-0.75,1) ;
\draw (\d,0) to [in=-135] (\d+0.75,1) to [in=\inang] (0,\h);
\draw (\d+1.5,0) to [in=-45] (\d+0.75,1) ;
\node[zxvertex=\zxwhite, zxdown] at (-\d-0.75,1){};
\node[zxvertex=\zxwhite, zxdown] at (\d+0.75,1){};
\node[box,zxdown] at (0,\h) {$\rho$};
\end{tz}
~~=~~
\def\htop{2.25}
\def\hbot{1.25}
\begin{tz}[zx,every to/.style={out=up, in=down}]
\draw (0,0) to (0,3);
\draw (-\d-1.5,0) to [in=-135] (0,\htop);
\draw (-\d,0) to [in=-135] (0,\hbot);
\draw (\d,0) to [in=-45] (0,\hbot);
\draw (\d+1.5,0) to [in=-45] (0,\htop);
\node[box,zxdown] at (0,\hbot) {$\rho$};
\node[box,zxdown] at (0,\htop) {$\rho$};
\end{tz}
&
\begin{tz}[zx, every to/.style={out=up, in=down}]
\draw (0,0) to (0,3);
\draw (-\d,1.2) to [in=-135] (0,\h);
\draw (\d,1.2) to [in=-45] (0,\h);
\node[box,zxdown] at (0,\h) {$\rho$};
\node[zxvertex=\zxwhite] at (-\d,1.2){};
\node[zxvertex=\zxwhite] at (\d,1.2){};
\end{tz}
~~=~~~
\begin{tz}[zx, every to/.style={out=up, in=down}]
\draw (0,0) to (0,3);
\end{tz}
&
\def\x{0.2}
\begin{tz}[zx, every to/.style={out=up, in=down},xscale=0.8]
\draw (0,0) to (0,3);
\draw (-\x,1.5) to [out=up, in=right] (-0.75-\x, 2.25) to [out=left, in=up] (-1.5-\x, 1.5) to (-1.5-\x,0);
\draw (\x,1.5) to [out=up, in=left] (0.75+\x, 2.25) to [out=right, in=up] (1.5+\x, 1.5) to (1.5+\x,0);
\node[zxvertex=\zxwhite] at (-0.75-\x, 2.25){};
\node[zxvertex=\zxwhite] at (0.75+\x, 2.25){};
\node[box] at (0,1.5) {$\rho^\dagger$};
\end{tz}
~~=~~
\begin{tz}[zx, every to/.style={out=up, in=down},xscale=0.8]
\draw (0,0) to (0,3);
\draw (-1.25, 0) to [in=-135] (0,1.95) ;
\draw (1.25,0) to [in=-45] (0,1.95);
\node[box] at (0,1.95) {$\rho$};
\end{tz}
\end{calign}
\end{definition}
\noindent
We usually denote an $A{-}B$-dagger bimodule $M$ by $_AM_B$.
\begin{definition} A \textit{morphism of dagger bimodules} $_AM_B\to {}_AN_B$ is a morphism $f:M\to N$ that commutes with the action of the Frobenius monoids:
\begin{calign}
\begin{tz}[zx]
\draw (0,0) to (0,3);
\draw (-1,0) to [out=up, in=-135] (0,2.15);
\draw (1,0) to [out=up, in=-45] (0,2.15) ;
\node[zxnode=\zxwhite] at (0,0.85) {$f$};
\node[boxvertex,zxdown] at (0,2.15){};
\end{tz}
=
\begin{tz}[zx]
\draw (0,0) to (0,3);
\draw (-1,0) to [out=up, in=-135] (0,0.85);
\draw (1,0) to [out=up, in=-45] (0,0.85) ;
\node[zxnode=\zxwhite] at (0,2.15) {$f$};
\node[boxvertex,zxdown] at (0,0.85){};
\end{tz}
\end{calign}
Two dagger bimodules are \textit{isomorphic}, here written ${}_AM_B\cong{}_AN_B$, if there is a unitary morphism of dagger bimodules ${}_AM_B\to{}_AN_B$.
\end{definition}
\noindent
In a monoidal dagger category in which dagger idempotents split, we can compose dagger bimodules ${}_AM_B$ and ${}_BN_C$ to obtain an $A{-}C$-dagger bimodule ${}_AM{\otimes_B}N_C$ as follows. First note that the following endomorphism is a dagger idempotent:
\begin{calign}\label{eq:idempotentforrelprod}
\begin{tz}[zx,every to/.style={out=up, in=down}]
\draw (0,0) to (0,3);
\draw (2,0) to (2,3);
\draw (0,2.25) to [out=-45, in=left] (1, 1.2) to[out=right, in=-135] (2,2.25);
\draw (-1,1.2) to[out=90,in=-135] (0,2.25);
\draw (3,1.2) to[out=90,in=-45] (2,2.25);
\node[zxvertex=\zxwhite] at (1,1.2){};
\node[zxvertex=\zxwhite] at (-1,1.2){};
\node[zxvertex=\zxwhite] at (3,1.2){};
\node[boxvertex,zxdown] at (0,2.25){};
\node[boxvertex,zxdown] at (2,2.25){};
\node[dimension, left] at (0,0) {$M$};
\node[dimension, right] at (2,0) {$N$};
\end{tz}
\end{calign}
The \emph{relative tensor product} ${}_AM{\otimes_B}N_C$ is defined as the image of the splitting of this idempotent. We depict the isometry $i: M\otimes_B N\to M\otimes N$ as a downwards pointing triangle:
\begin{calign}\label{eq:moritaidempotentsplit}
\begin{tz}[zx,every to/.style={out=up, in=down}]
\draw (0,0) to (0,3);
\draw (2,0) to (2,3);
\draw (0,2.25) to [out=-45, in=left] (1, 1.2) to[out=right, in=-135] (2,2.25);
\draw (-1,1.2) to[out=90,in=-135] (0,2.25);
\draw (3,1.2) to[out=90,in=-45] (2,2.25);
\node[zxvertex=\zxwhite] at (1,1.2){};
\node[zxvertex=\zxwhite] at (-1,1.2) {};
\node[zxvertex=\zxwhite] at (3,1.2){};
\node[boxvertex,zxdown] at (0,2.25){};
\node[boxvertex,zxdown] at (2,2.25){};
\end{tz}
~~=~~
\begin{tz}[zx,every to/.style={out=up, in=down}]
\draw (0,0) to (0,0.5);
\draw (0,2.5) to (0,3);
\draw (2,0) to (2,0.5);
\draw (2,2.5) to (2,3);
\draw (1,0.5) to (1,2.5);
\node[dimension, right] at (1,1.5) {$M{\otimes_B}N$};
\node[triangleup=2] at (1,0.5){};
\node[triangledown=2] at (1,2.5){};
\end{tz}
&
\begin{tz}[zx]
\clip (-1.2, -0.3) rectangle (1.9,3.3);
\draw (0,0) to (0,1);
\draw (-1,1) to (-1,2);
\draw (1,1) to (1,2);
\draw (0, 2) to (0,3);
\node[triangleup=2] at (0,2){};
\node[triangledown=2] at (0,1){};
\node[dimension, right] at (0,0) {$M{\otimes_B}N$};
\end{tz}
=~~
\begin{tz}[zx]
\clip (-0.2, -0.3) rectangle (1.9,3.3);
\draw (0,0) to (0,3);
\node[dimension, right] at (0,0) {$M{\otimes_B}N$};
\end{tz}
\end{calign}
\ignore{One can convince onself that $f:M\otimes N\to M\otimes_B N$ is a coequalizer for the }
For dagger bimodules ${}_AM_B$ and ${}_BN_C$, the relative tensor product $M\otimes_B N$ is itself an $A{-}C$-dagger bimodule with the following action $A\otimes(M{\otimes_B}N) \otimes C\to M{\otimes_B}N$:
\begin{equation}
\begin{tz}[zx,every to/.style={out=up, in=down}]
\draw (0,-0.) to (0,1) ;
\draw (0,3) to (0,3.5);
\draw (-1,1) to (-1,2.5);
\draw (1,1) to (1,2.5);
\draw (-2,-0.) to [in=-135] (-1,1.75);
\draw (2,-0.) to [in=-45] (1,1.75);
\node[triangledown=2] at (0,1){};
\node[triangleup=2] at (0,2.5){};
\node[boxvertex] at (-1,1.75){};
\node[boxvertex] at (1,1.75){};
\end{tz}
\end{equation}
\noindent
\begin{definition} \label{def:daggermoritaequiv}Two special Frobenius monoids $A$ and $B$ are \textit{Morita equivalent} if there are dagger bimodules $_AM_B$ and $_BN_A$ such that $_AM{\otimes_B}N_A\cong {}_AA_A$ and ${_BN{\otimes_A}M_B \cong {}_BB_B}$.
\end{definition}
\noindent
We make use of the following result.
\begin{theorem}[{\cite[Thm. A.1]{Musto2019}}]
Let $\mathcal{C}$ be a dagger 2-category in which all dagger idempotents split and let $X: r \to s$ and $Y: r \to t$ be special 1-morphisms. Then the special Frobenius monoids $X \otimes X^*$ and $Y \otimes Y^*$ in $\End(r)$ are Morita equivalent if and only if $s$ is dagger equivalent to $t$.
\end{theorem}
\begin{proof}
The proof is identical to that of~\cite[Thm. A.1]{Musto2019}, which classifies objects from which there is a morphism \emph{into} $a$; one need only read the diagrams from left to right rather than from right to left.
\end{proof}
\noindent
We can equate dagger equivalence of objects in $\Fun_u(\Rep(G),\Hilb)$ to a more familiar notion. 
\begin{lemma}
In $\Fun_u(\Rep(G),\Hilb)$ there exists a dagger equivalence between two objects $F_1,F_2$ iff these functors are unitarily monoidally naturally isomorphic.
\end{lemma}
\begin{proof}
For a pseudonatural transformation $(\alpha,H): F_1 \to F_2$ to be a dagger equivalence in $\Fun_u(\Rep(G),\Hilb)$, there must exist a pseudonatural transformation $(\alpha^{-1},K): G \to F$ and an unitary isomorphism $f: \mathbb{C} \to H \otimes K$. But then $H$ must be 1-dimensional, and therefore unitarily isomorphic to the unit object $\mathbb{C}$. Conjugating $(\alpha,H)$ by this isomorphism, we obtain a unitary monoidal natural isomorphism $F_1 \to F_2$. In the other direction, a unitary monoidal natural isomorphism is clearly a dagger equivalence; the weak inverse is the actual inverse and the unitary 2-morphisms witnessing the equivalence are trivial.  \end{proof}
\noindent
Putting these results together, we obtain the following classification. 
\begin{corollary}
Let $G$ be a compact quantum group and let $F: \Rep(G) \to \Hilb$ be the canonical fibre functor. There is a bijective correspondence between the following structures:
\begin{itemize}
\item Unitary monoidal natural isomorphism classes of fibre functors accessible from $F$ by a UPT.
\item Morita equivalence classes of simple Frobenius monoids in $\End(F) \cong \Rep(A_G)$.
\end{itemize}
\end{corollary}

\bibliographystyle{alphaurl}
\bibliography{bibliography}

\begin{thebibliography}{MRV19}

\bibitem[Hum12]{Hummon2012}
Benjamin~Taylor Hummon.
\newblock {\em Surface diagrams for {G}ray-categories}.
\newblock PhD thesis, UC San Diego, 2012.
\newblock URL: \url{https://escholarship.org/uc/item/5b24s9cc}.

\bibitem[HV19]{Heunen2019}
Chris Heunen and Jamie Vicary.
\newblock {\em Categories for Quantum Theory: An Introduction}.
\newblock Oxford Graduate Texts in Mathematics Series. Oxford University Press,
  2019.
\newblock \href {http://dx.doi.org/10.1093/oso/9780198739623.001.0001}
  {\path{doi:10.1093/oso/9780198739623.001.0001}}.

\bibitem[Lac10]{Lack2010}
Stephen Lack.
\newblock Icons.
\newblock {\em Applied Categorical Structures}, 18(3):289--307, 2010.
\newblock \href {http://arxiv.org/abs/0711.4657} {\path{arXiv:0711.4657}}.

\bibitem[Lei98]{Leinster1998}
Tom Leinster.
\newblock Basic bicategories.
\newblock 1998.
\newblock \href {http://arxiv.org/abs/math/9810017}
  {\path{arXiv:math/9810017}}.

\bibitem[Mac63]{MacLane1963}
Saunders MacLane.
\newblock Natural associativity and commutativity.
\newblock {\em Rice Institute Pamphlet-Rice University Studies}, 49(4), 1963.

\bibitem[Mar14]{Marsden2014}
Daniel Marsden.
\newblock Category theory using string diagrams.
\newblock 2014.
\newblock \href {http://arxiv.org/abs/1401.7220} {\path{arXiv:1401.7220}}.

\bibitem[Mel06]{Mellies2006}
Paul-Andr{\'e} Melli{\`e}s.
\newblock Functorial boxes in string diagrams.
\newblock In {\em International Workshop on Computer Science Logic}, pages
  1--30. Springer, 2006.
\newblock URL:
  \url{https://www.irif.fr/~mellies/mpri/mpri-ens/articles/mellies-functorial-boxes.pdf}.

\bibitem[MRV19]{Musto2019}
Benjamin Musto, David Reutter, and Dominic Verdon.
\newblock The {M}orita theory of quantum graph isomorphisms.
\newblock {\em Communications in Mathematical Physics}, 365(2):797--845, 2019.
\newblock \href {http://arxiv.org/abs/1801.09705} {\path{arXiv:1801.09705}}.

\bibitem[Sel10]{Selinger2010}
Peter Selinger.
\newblock A survey of graphical languages for monoidal categories.
\newblock In {\em New {S}tructures for {P}hysics}, pages 289--355. Springer,
  2010.
\newblock \href {http://arxiv.org/abs/0908.3347} {\path{arXiv:0908.3347}}.

\bibitem[TV17]{Turaev2017}
V.~Turaev and A.~Virelizier.
\newblock {\em Monoidal Categories and Topological Field Theory}.
\newblock Progress in Mathematics. Springer International Publishing, 2017.

\bibitem[Ver22]{Verdon2020}
Dominic Verdon.
\newblock Unitary transformations of fibre functors.
\newblock {\em Journal of Pure and Applied Algebra}, 226(7), July 2022.
\newblock \href {http://arxiv.org/abs/2004.12761} {\path{arXiv:2004.12761}},
  \href {http://dx.doi.org/10.1016/j.jpaa.2021.106989}
  {\path{doi:10.1016/j.jpaa.2021.106989}}.

\bibitem[Vic12]{Vicary2012}
Jamie Vicary.
\newblock Higher quantum theory.
\newblock 2012.
\newblock \href {http://arxiv.org/abs/1207.4563} {\path{arXiv:1207.4563}}.

\end{thebibliography}

\end{document}